\documentclass[draft, 12pt]{amsart}

\usepackage{amssymb}
\usepackage{amsmath}

\usepackage{color}

\makeatletter
 
\def\diagram{\m@th\leftwidth=\z@ \rightwidth=\z@ \topheight=\z@
\botheight=\z@ \setbox\@picbox\hbox\bgroup}
 
\def\enddiagram{\egroup\wd\@picbox\rightwidth\unitlength
\ht\@picbox\topheight\unitlength \dp\@picbox\botheight\unitlength
\hskip\leftwidth\unitlength\box\@picbox}
 
\def\bfig{\begin{diagram}}
\def\efig{\end{diagram}}
\newcount\wideness \newcount\leftwidth \newcount\rightwidth
\newcount\highness \newcount\topheight \newcount\botheight
 
\def\ratchet#1#2{\ifnum#1<#2 \global #1=#2 \fi}
 
\def\putbox(#1,#2)#3{%
\horsize{\wideness}{#3} \divide\wideness by 2
{\advance\wideness by #1 \ratchet{\rightwidth}{\wideness}}
{\advance\wideness by -#1 \ratchet{\leftwidth}{\wideness}}
\vertsize{\highness}{#3} \divide\highness by 2
{\advance\highness by #2 \ratchet{\topheight}{\highness}}
{\advance\highness by -#2 \ratchet{\botheight}{\highness}}
\put(#1,#2){\makebox(0,0){$#3$}}}
 
\def\putlbox(#1,#2)#3{%
\horsize{\wideness}{#3}
{\advance\wideness by #1 \ratchet{\rightwidth}{\wideness}}
{\ratchet{\leftwidth}{-#1}}
\vertsize{\highness}{#3} \divide\highness by 2
{\advance\highness by #2 \ratchet{\topheight}{\highness}}
{\advance\highness by -#2 \ratchet{\botheight}{\highness}}
\put(#1,#2){\makebox(0,0)[l]{$#3$}}}
 
\def\putrbox(#1,#2)#3{%
\horsize{\wideness}{#3}
{\ratchet{\rightwidth}{#1}}
{\advance\wideness by -#1 \ratchet{\leftwidth}{\wideness}}
\vertsize{\highness}{#3} \divide\highness by 2
{\advance\highness by #2 \ratchet{\topheight}{\highness}}
{\advance\highness by -#2 \ratchet{\botheight}{\highness}}
\put(#1,#2){\makebox(0,0)[r]{$#3$}}}

\def\adjust[#1]{} 
 
\newcount \coefa
\newcount \coefb
\newcount \coefc
\newcount\tempcounta
\newcount\tempcountb
\newcount\tempcountc
\newcount\tempcountd
\newcount\xext
\newcount\yext
\newcount\xoff
\newcount\yoff
\newcount\gap%
\newcount\arrowtypea
\newcount\arrowtypeb
\newcount\arrowtypec
\newcount\arrowtyped
\newcount\arrowtypee
\newcount\height
\newcount\width
\newcount\xpos
\newcount\ypos
\newcount\run
\newcount\rise
\newcount\arrowlength
\newcount\halflength
\newcount\arrowtype
\newdimen\tempdimen
\newdimen\xlen
\newdimen\ylen
\newsavebox{\tempboxa}%
\newsavebox{\tempboxb}%
\newsavebox{\tempboxc}%
 
\newdimen\w@dth
 
\def\setw@dth#1#2{\setbox\z@\hbox{\m@th$#1$}\w@dth=\wd\z@
\setbox\@ne\hbox{\m@th$#2$}\ifnum\w@dth<\wd\@ne \w@dth=\wd\@ne \fi
\advance\w@dth by 1.2em}
 
\def\t@^#1_#2{\allowbreak\def\n@one{#1}\def\n@two{#2}\mathrel
{\setw@dth{#1}{#2}
\mathop{\hbox to \w@dth{\rightarrowfill}}\limits
\ifx\n@one\empty\else ^{\box\z@}\fi
\ifx\n@two\empty\else _{\box\@ne}\fi}}
\def\t@@^#1{\@ifnextchar_{\t@^{#1}}{\t@^{#1}_{}}}
\def\to{\@ifnextchar^{\t@@}{\t@@^{}}}
 
\def\t@left^#1_#2{\def\n@one{#1}\def\n@two{#2}\mathrel{\setw@dth{#1}{#2}
\mathop{\hbox to \w@dth{\leftarrowfill}}\limits
\ifx\n@one\empty\else ^{\box\z@}\fi
\ifx\n@two\empty\else _{\box\@ne}\fi}}
\def\t@@left^#1{\@ifnextchar_{\t@left^{#1}}{\t@left^{#1}_{}}}
\def\toleft{\@ifnextchar^{\t@@left}{\t@@left^{}}}
 
\def\two@^#1_#2{\allowbreak
\def\n@one{#1}\def\n@two{#2}\mathrel{\setw@dth{#1}{#2}
\mathop{\vcenter{\lineskip\z@\baselineskip\z@
                 \hbox to \w@dth{\rightarrowfill}%
                 \hbox to \w@dth{\rightarrowfill}}%
       }\limits
\ifx\n@one\empty\else ^{\box\z@}\fi
\ifx\n@two\empty\else _{\box\@ne}\fi}}
\def\tw@@^#1{\@ifnextchar _{\two@^{#1}}{\two@^{#1}_{}}}
\def\two{\@ifnextchar ^{\tw@@}{\tw@@^{}}}
 
\def\tofr@^#1_#2{\def\n@one{#1}\def\n@two{#2}\mathrel{\setw@dth{#1}{#2}
\mathop{\vcenter{\hbox to \w@dth{\rightarrowfill}\kern-1.7ex
                 \hbox to \w@dth{\leftarrowfill}}%
       }\limits
\ifx\n@one\empty\else ^{\box\z@}\fi
\ifx\n@two\empty\else _{\box\@ne}\fi}}
\def\t@fr@^#1{\@ifnextchar_ {\tofr@^{#1}}{\tofr@^{#1}_{}}}
\def\tofro{\@ifnextchar^ {\t@fr@}{\t@fr@^{}}}

\def\mon{\mathop{\m@th\hbox to
      14.6\P@{\lasyb\char'51\hskip-2.1\P@$\arrext$\hss
$\mathord\rightarrow$}}\limits} 
\def\leftmono{\mathrel{\m@th\hbox to
14.6\P@{$\mathord\leftarrow$\hss$\arrext$\hskip-2.1\P@\lasyb\char'50%
}}\limits} 
\mathchardef\arrext="0200       

\def\settypes(#1,#2,#3){\arrowtypea#1 \arrowtypeb#2 \arrowtypec#3}
\def\settoheight#1#2{\setbox\@tempboxa\hbox{#2}#1\ht\@tempboxa\relax}%
\def\settodepth#1#2{\setbox\@tempboxa\hbox{#2}#1\dp\@tempboxa\relax}%
\def\settokens`#1`#2`#3`#4`{%
     \def\tokena{#1}\def\tokenb{#2}\def\tokenc{#3}\def\tokend{#4}}
\def\setsqparms[#1`#2`#3`#4;#5`#6]{%
\arrowtypea #1
\arrowtypeb #2
\arrowtypec #3
\arrowtyped #4
\width #5
\height #6
}
\def\setpos(#1,#2){\xpos=#1 \ypos#2}

\def\settriparms[#1`#2`#3;#4]{\settripairparms[#1`#2`#3`1`1;#4]}%
 
\def\settripairparms[#1`#2`#3`#4`#5;#6]{%
\arrowtypea #1
\arrowtypeb #2
\arrowtypec #3
\arrowtyped #4
\arrowtypee #5
\width #6
\height #6
}
 
\def\resetparms{\settripairparms[1`1`1`1`1;500]\width 500}
 
\resetparms
 
\def\mvector(#1,#2)#3{
\put(0,0){\vector(#1,#2){#3}}%
\put(0,0){\vector(#1,#2){26}}%
}
\def\evector(#1,#2)#3{{
\arrowlength #3
\put(0,0){\vector(#1,#2){\arrowlength}}%
\advance \arrowlength by-30
\put(0,0){\vector(#1,#2){\arrowlength}}%
}}
 
\def\horsize#1#2{%
\settowidth{\tempdimen}{$#2$}%
#1=\tempdimen
\divide #1 by\unitlength
}
 
\def\vertsize#1#2{%
\settoheight{\tempdimen}{$#2$}%
#1=\tempdimen
\settodepth{\tempdimen}{$#2$}%
\advance #1 by\tempdimen
\divide #1 by\unitlength
}
 
\def\putvector(#1,#2)(#3,#4)#5#6{{%
\ifnum3<\arrowtype
\putdashvector(#1,#2)(#3,#4)#5\arrowtype
\else
\ifnum\arrowtype<-3
\putdashvector(#1,#2)(#3,#4)#5\arrowtype
\else
\xpos=#1
\ypos=#2
\run=#3
\rise=#4
\arrowlength=#5
\ifnum \arrowtype<0
    \ifnum \run=0
        \advance \ypos by-\arrowlength
    \else
        \tempcounta \arrowlength
        \multiply \tempcounta by\rise
        \divide \tempcounta by\run
        \ifnum\run>0
            \advance \xpos by\arrowlength
            \advance \ypos by\tempcounta
        \else
            \advance \xpos by-\arrowlength
            \advance \ypos by-\tempcounta
        \fi
    \fi
    \multiply \arrowtype by-1
    \multiply \rise by-1
    \multiply \run by-1
\fi
\ifcase \arrowtype
\or \put(\xpos,\ypos){\vector(\run,\rise){\arrowlength}}%
\or \put(\xpos,\ypos){\mvector(\run,\rise)\arrowlength}%
\or \put(\xpos,\ypos){\evector(\run,\rise){\arrowlength}}%
\fi\fi\fi
}}
 
\def\putsplitvector(#1,#2)#3#4{
\xpos #1
\ypos #2
\arrowtype #4
\halflength #3
\arrowlength #3
\gap 140
\advance \halflength by-\gap
\divide \halflength by2
\ifnum\arrowtype>0
   \ifcase \arrowtype
   \or \put(\xpos,\ypos){\line(0,-1){\halflength}}%
       \advance\ypos by-\halflength
       \advance\ypos by-\gap
       \put(\xpos,\ypos){\vector(0,-1){\halflength}}%
   \or \put(\xpos,\ypos){\line(0,-1)\halflength}%
       \put(\xpos,\ypos){\vector(0,-1)3}%
       \advance\ypos by-\halflength
       \advance\ypos by-\gap
       \put(\xpos,\ypos){\vector(0,-1){\halflength}}%
   \or \put(\xpos,\ypos){\line(0,-1)\halflength}%
       \advance\ypos by-\halflength
       \advance\ypos by-\gap
       \put(\xpos,\ypos){\evector(0,-1){\halflength}}%
   \fi
\else \arrowtype=-\arrowtype
   \ifcase\arrowtype
   \or \advance \ypos by-\arrowlength
       \put(\xpos,\ypos){\line(0,1){\halflength}}%
       \advance\ypos by\halflength
       \advance\ypos by\gap
       \put(\xpos,\ypos){\vector(0,1){\halflength}}%
   \or \advance \ypos by-\arrowlength
       \put(\xpos,\ypos){\line(0,1)\halflength}%
       \put(\xpos,\ypos){\vector(0,1)3}%
       \advance\ypos by\halflength
       \advance\ypos by\gap
       \put(\xpos,\ypos){\vector(0,1){\halflength}}%
   \or \advance \ypos by-\arrowlength
       \put(\xpos,\ypos){\line(0,1)\halflength}%
       \advance\ypos by\halflength
       \advance\ypos by\gap
       \put(\xpos,\ypos){\evector(0,1){\halflength}}%
   \fi
\fi
}
 
\def\putmorphism(#1)(#2,#3)[#4`#5`#6]#7#8#9{{%
\run #2
\rise #3
\ifnum\rise=0
  \puthmorphism(#1)[#4`#5`#6]{#7}{#8}#9%
\else\ifnum\run=0
  \putvmorphism(#1)[#4`#5`#6]{#7}{#8}#9%
\else
\setpos(#1)%
\arrowlength #7
\arrowtype #8
\ifnum\run=0
\else\ifnum\rise=0
\else
\ifnum\run>0
    \coefa=1
\else
   \coefa=-1
\fi
\ifnum\arrowtype>0
   \coefb=0
   \coefc=-1
\else
   \coefb=\coefa
   \coefc=1
   \arrowtype=-\arrowtype
\fi
\width=2
\multiply \width by\run
\divide \width by\rise
\ifnum \width<0  \width=-\width\fi
\advance\width by60
\if l#9 \width=-\width\fi
\putbox(\xpos,\ypos){#4}
{\multiply \coefa by\arrowlength
\advance\xpos by\coefa
\multiply \coefa by\rise
\divide \coefa by\run
\advance \ypos by\coefa
\putbox(\xpos,\ypos){#5} }%
{\multiply \coefa by\arrowlength
\divide \coefa by2
\advance \xpos by\coefa
\advance \xpos by\width
\multiply \coefa by\rise
\divide \coefa by\run
\advance \ypos by\coefa
\if l#9%
   \putrbox(\xpos,\ypos){#6}%
\else\if r#9%
   \putlbox(\xpos,\ypos){#6}%
\fi\fi }%
{\multiply \rise by-\coefc
\multiply \run by-\coefc
\multiply \coefb by\arrowlength
\advance \xpos by\coefb
\multiply \coefb by\rise
\divide \coefb by\run
\advance \ypos by\coefb
\multiply \coefc by70
\advance \ypos by\coefc
\multiply \coefc by\run
\divide \coefc by\rise
\advance \xpos by\coefc
\multiply \coefa by140
\multiply \coefa by\run
\divide \coefa by\rise
\advance \arrowlength by\coefa
\ifcase\arrowtype
\or \put(\xpos,\ypos){\vector(\run,\rise){\arrowlength}}%
\or \put(\xpos,\ypos){\mvector(\run,\rise){\arrowlength}}%
\or \put(\xpos,\ypos){\evector(\run,\rise){\arrowlength}}%
\fi}\fi\fi\fi\fi}}

\newcount\numbdashes \newcount\lengthdash \newcount\increment
 
\def\howmanydashes{
\numbdashes=\arrowlength \lengthdash=40
\divide\numbdashes by \lengthdash
\lengthdash=\arrowlength
\divide\lengthdash by \numbdashes
\increment=\lengthdash
\multiply\lengthdash by 3
\divide\lengthdash by 5
}
 
\def\putdashvector(#1)(#2,#3)#4#5{%
\ifnum#3=0 \putdashhvector(#1){#4}#5
\else
\ifnum#2=0
\putdashvvector(#1){#4}#5\fi\fi}
 
\def\putdashhvector(#1,#2)#3#4{{%
\arrowlength=#3 \howmanydashes
\multiput(#1,#2)(\increment,0){\numbdashes}%
{\vrule height .4pt width \lengthdash\unitlength}
\arrowtype=#4 \xpos=#1
\ifnum\arrowtype<0 \advance\arrowtype by 7 \fi
\ifcase\arrowtype
\or \advance\xpos by 10
    \put(\xpos,#2){\vector(-1,0){\lengthdash}}
    \advance\xpos by 40
    \put(\xpos,#2){\vector(-1,0){\lengthdash}}
\or \advance \xpos by 10
    \put(\xpos,#2){\vector(-1,0){\lengthdash}}
    \advance\xpos by  \arrowlength
    \advance\xpos by  -50
    \put(\xpos,#2){\vector(-1,0){\lengthdash}}
\or \advance\xpos by 10
    \put(\xpos,#2){\vector(-1,0){\lengthdash}}
\or \advance\xpos by \arrowlength
    \advance\xpos by -\lengthdash
    \put(\xpos,#2){\vector(1,0){\lengthdash}}
\or {\advance\xpos by 10
    \put(\xpos,#2){\vector(1,0){\lengthdash}}}
    \advance\xpos by \arrowlength
    \advance\xpos by -\lengthdash
    \put(\xpos,#2){\vector(1,0){\lengthdash}}
\or \advance\xpos by \arrowlength
    \advance\xpos by -\lengthdash
    \put(\xpos,#2){\vector(1,0){\lengthdash}}
    \advance\xpos by -40
    \put(\xpos,#2){\vector(1,0){\lengthdash}}
   \fi
}}
 
\def\putdashvvector(#1,#2)#3#4{{%
\arrowlength=#3 \howmanydashes
\ypos=#2 \advance\ypos by -\arrowlength
\multiput(#1,#2)(0,\increment){\numbdashes}%
    {\vrule width .4pt height \lengthdash\unitlength}
\arrowtype=#4 \ypos=#2
\ifnum\arrowtype<0 \advance\arrowtype by 7 \fi
\ifcase\arrowtype
\or \advance\ypos by \arrowlength \advance\ypos by -40
    \put(#1,\ypos){\vector(0,1){\lengthdash}}
    \advance\ypos by -40
    \put(#1,\ypos){\vector(0,1){\lengthdash}}
\or \advance\ypos by 10
    \put(#1,\ypos){\vector(0,1){\lengthdash}}
    \advance\ypos by \arrowlength \advance\ypos by -40
    \put(#1,\ypos){\vector(0,1){\lengthdash}}
\or \advance\ypos by \arrowlength \advance\ypos by -40
    \put(#1,\ypos){\vector(0,1){\lengthdash}}
\or \advance\ypos by 10
    \put(#1,\ypos){\vector(0,-1){\lengthdash}}
\or \advance\ypos by 10
    \put(#1,\ypos){\vector(0,-1){\lengthdash}}
    \advance\ypos by \arrowlength \advance\ypos by -40
    \put(#1,\ypos){\vector(0,-1){\lengthdash}}
\or \advance\ypos by 10
    \put(#1,\ypos){\vector(0,-1){\lengthdash}}
    \advance\ypos by 40
    \put(#1,\ypos){\vector(0,-1){\lengthdash}}
\fi
}}
 
\def\puthmorphism(#1,#2)[#3`#4`#5]#6#7#8{{%
\xpos #1
\ypos #2
\width #6
\arrowlength #6
\arrowtype=#7
\putbox(\xpos,\ypos){#3\vphantom{#4}}%
{\advance \xpos by\arrowlength
\putbox(\xpos,\ypos){\vphantom{#3}#4}}%
\horsize{\tempcounta}{#3}%
\horsize{\tempcountb}{#4}%
\divide \tempcounta by2
\divide \tempcountb by2
\advance \tempcounta by30
\advance \tempcountb by30
\advance \xpos by\tempcounta
\advance \arrowlength by-\tempcounta
\advance \arrowlength by-\tempcountb
\putvector(\xpos,\ypos)(1,0)\arrowlength\arrowtype
\divide \arrowlength by2
\advance \xpos by\arrowlength
\vertsize{\tempcounta}{#5}%
\divide\tempcounta by2
\advance \tempcounta by20
\if a#8 %
   \advance \ypos by\tempcounta
   \putbox(\xpos,\ypos){#5}%
\else
   \advance \ypos by-\tempcounta
   \putbox(\xpos,\ypos){#5}%
\fi}}
 
\def\putvmorphism(#1,#2)[#3`#4`#5]#6#7#8{{%
\xpos #1
\ypos #2
\arrowlength #6
\arrowtype #7
\settowidth{\xlen}{$#5$}%
\putbox(\xpos,\ypos){#3}%
{\advance \ypos by-\arrowlength
\putbox(\xpos,\ypos){#4}}%
{\advance\arrowlength by-140
\advance \ypos by-70
\ifdim\xlen>0pt
   \if m#8%
      \putsplitvector(\xpos,\ypos)\arrowlength\arrowtype
   \else
   \putvector(\xpos,\ypos)(0,-1)\arrowlength\arrowtype
   \fi
\else
   \putvector(\xpos,\ypos)(0,-1)\arrowlength\arrowtype
\fi}%
\ifdim\xlen>0pt
   \divide \arrowlength by2
   \advance\ypos by-\arrowlength
   \if l#8%
      \advance \xpos by-40
      \putrbox(\xpos,\ypos){#5}%
   \else\if r#8%
      \advance \xpos by40
      \putlbox(\xpos,\ypos){#5}%
   \else
      \putbox(\xpos,\ypos){#5}%
   \fi\fi
\fi
}}
 
\def\putsquarep<#1>(#2)[#3;#4`#5`#6`#7]{{%
\setsqparms[#1]%
\setpos(#2)%
\settokens`#3`%
\puthmorphism(\xpos,\ypos)[\tokenc`\tokend`{#7}]{\width}{\arrowtyped}b%
\advance\ypos by \height
\puthmorphism(\xpos,\ypos)[\tokena`\tokenb`{#4}]{\width}{\arrowtypea}a%
\putvmorphism(\xpos,\ypos)[``{#5}]{\height}{\arrowtypeb}l%
\advance\xpos by \width
\putvmorphism(\xpos,\ypos)[``{#6}]{\height}{\arrowtypec}r%
}}
 
\def\putsquare{\@ifnextchar <{\putsquarep}{\putsquarep%
   <\arrowtypea`\arrowtypeb`\arrowtypec`\arrowtyped;\width`\height>}}
\def\square{\@ifnextchar< {\squarep}{\squarep
   <\arrowtypea`\arrowtypeb`\arrowtypec`\arrowtyped;\width`\height>}}
\def\squarep<#1>[#2`#3`#4`#5;#6`#7`#8`#9]{{
\setsqparms[#1]
\diagram
\putsquarep<\arrowtypea`\arrowtypeb`\arrowtypec`
\arrowtyped;\width`\height>
(0,0)[#2`#3`#4`{#5};#6`#7`#8`{#9}]
\enddiagram
}}                                                 
\def\putptrianglep<#1>(#2,#3)[#4`#5`#6;#7`#8`#9]{{%
\settriparms[#1]%
\xpos=#2 \ypos=#3
\advance\ypos by \height
\puthmorphism(\xpos,\ypos)[#4`#5`{#7}]{\height}{\arrowtypea}a%
\putvmorphism(\xpos,\ypos)[`#6`{#8}]{\height}{\arrowtypeb}l%
\advance\xpos by\height
\putmorphism(\xpos,\ypos)(-1,-1)[``{#9}]{\height}{\arrowtypec}r%
}}
 
\def\putptriangle{\@ifnextchar <{\putptrianglep}{\putptrianglep
   <\arrowtypea`\arrowtypeb`\arrowtypec;\height>}}
\def\ptriangle{\@ifnextchar <{\ptrianglep}{\ptrianglep
   <\arrowtypea`\arrowtypeb`\arrowtypec;\height>}}
\def\ptrianglep<#1>[#2`#3`#4;#5`#6`#7]{{
\settriparms[#1]
\diagram
\putptrianglep<\arrowtypea`\arrowtypeb`
\arrowtypec;\height>
(0,0)[#2`#3`#4;#5`#6`{#7}]
\enddiagram
}}                                            
 
\def\putqtrianglep<#1>(#2,#3)[#4`#5`#6;#7`#8`#9]{{%
\settriparms[#1]%
\xpos=#2 \ypos=#3
\advance\ypos by\height
\puthmorphism(\xpos,\ypos)[#4`#5`{#7}]{\height}{\arrowtypea}a%
\putmorphism(\xpos,\ypos)(1,-1)[``{#8}]{\height}{\arrowtypeb}l%
\advance\xpos by\height
\putvmorphism(\xpos,\ypos)[`#6`{#9}]{\height}{\arrowtypec}r%
}}
 
\def\putqtriangle{\@ifnextchar <{\putqtrianglep}{\putqtrianglep
   <\arrowtypea`\arrowtypeb`\arrowtypec;\height>}}
\def\qtriangle{\@ifnextchar <{\qtrianglep}{\qtrianglep
   <\arrowtypea`\arrowtypeb`\arrowtypec;\height>}}
\def\qtrianglep<#1>[#2`#3`#4;#5`#6`#7]{{
\settriparms[#1]
\width=\height                                
\diagram
\putqtrianglep<\arrowtypea`\arrowtypeb`
\arrowtypec;\height>
(0,0)[#2`#3`#4;#5`#6`{#7}]
\enddiagram
}}
 
\def\putdtrianglep<#1>(#2,#3)[#4`#5`#6;#7`#8`#9]{{%
\settriparms[#1]%
\xpos=#2 \ypos=#3
\puthmorphism(\xpos,\ypos)[#5`#6`{#9}]{\height}{\arrowtypec}b%
\advance\xpos by \height \advance\ypos by\height
\putmorphism(\xpos,\ypos)(-1,-1)[``{#7}]{\height}{\arrowtypea}l%
\putvmorphism(\xpos,\ypos)[#4``{#8}]{\height}{\arrowtypeb}r%
}}
 
\def\putdtriangle{\@ifnextchar <{\putdtrianglep}{\putdtrianglep
   <\arrowtypea`\arrowtypeb`\arrowtypec;\height>}}
\def\dtriangle{\@ifnextchar <{\dtrianglep}{\dtrianglep
   <\arrowtypea`\arrowtypeb`\arrowtypec;\height>}}
\def\dtrianglep<#1>[#2`#3`#4;#5`#6`#7]{{
\settriparms[#1]
\width=\height                                
\diagram
\putdtrianglep<\arrowtypea`\arrowtypeb`
\arrowtypec;\height>
(0,0)[#2`#3`#4;#5`#6`{#7}]
\enddiagram
}}
 
\def\putbtrianglep<#1>(#2,#3)[#4`#5`#6;#7`#8`#9]{{%
\settriparms[#1]%
\xpos=#2 \ypos=#3
\puthmorphism(\xpos,\ypos)[#5`#6`{#9}]{\height}{\arrowtypec}b%
\advance\ypos by\height
\putmorphism(\xpos,\ypos)(1,-1)[``{#8}]{\height}{\arrowtypeb}r%
\putvmorphism(\xpos,\ypos)[#4``{#7}]{\height}{\arrowtypea}l%
}}
 
\def\putbtriangle{\@ifnextchar <{\putbtrianglep}{\putbtrianglep
   <\arrowtypea`\arrowtypeb`\arrowtypec;\height>}}
\def\btriangle{\@ifnextchar <{\btrianglep}{\btrianglep
   <\arrowtypea`\arrowtypeb`\arrowtypec;\height>}}
\def\btrianglep<#1>[#2`#3`#4;#5`#6`#7]{{
\settriparms[#1]
\width=\height                               
\diagram
\putbtrianglep<\arrowtypea`\arrowtypeb`
\arrowtypec;\height>
(0,0)[#2`#3`#4;#5`#6`{#7}]
\enddiagram
}}
 
\def\putAtrianglep<#1>(#2,#3)[#4`#5`#6;#7`#8`#9]{{%
\settriparms[#1]%
\xpos=#2 \ypos=#3
{\multiply \height by2
\puthmorphism(\xpos,\ypos)[#5`#6`{#9}]{\height}{\arrowtypec}b}%
\advance\xpos by\height \advance\ypos by\height
\putmorphism(\xpos,\ypos)(-1,-1)[#4``{#7}]{\height}{\arrowtypea}l%
\putmorphism(\xpos,\ypos)(1,-1)[``{#8}]{\height}{\arrowtypeb}r%
}}
 
\def\putAtriangle{\@ifnextchar <{\putAtrianglep}{\putAtrianglep
   <\arrowtypea`\arrowtypeb`\arrowtypec;\height>}}
\def\Atriangle{\@ifnextchar <{\Atrianglep}{\Atrianglep
   <\arrowtypea`\arrowtypeb`\arrowtypec;\height>}}
\def\Atrianglep<#1>[#2`#3`#4;#5`#6`#7]{{
\settriparms[#1]
\width=\height                                     
\diagram
\putAtrianglep<\arrowtypea`\arrowtypeb`
\arrowtypec;\height>
(0,0)[#2`#3`#4;#5`#6`{#7}]
\enddiagram
}}
 
\def\putAtrianglepairp<#1>(#2)[#3;#4`#5`#6`#7`#8]{{%
\settripairparms[#1]%
\setpos(#2)%
\settokens`#3`%
\puthmorphism(\xpos,\ypos)[\tokenb`\tokenc`{#7}]{\height}{\arrowtyped}b%
\advance\xpos by\height
\puthmorphism(\xpos,\ypos)[\phantom{\tokenc}`\tokend`{#8}]%
{\height}{\arrowtypee}b%
\advance\ypos by\height
\putmorphism(\xpos,\ypos)(-1,-1)[\tokena``{#4}]{\height}{\arrowtypea}l%
\putvmorphism(\xpos,\ypos)[``{#5}]{\height}{\arrowtypeb}m%
\putmorphism(\xpos,\ypos)(1,-1)[``{#6}]{\height}{\arrowtypec}r%
}}
 
\def\putAtrianglepair{\@ifnextchar <{\putAtrianglepairp}{\putAtrianglepairp%
   <\arrowtypea`\arrowtypeb`\arrowtypec`\arrowtyped`\arrowtypee;\height>}}
\def\Atrianglepair{\@ifnextchar <{\Atrianglepairp}{\Atrianglepairp%
   <\arrowtypea`\arrowtypeb`\arrowtypec`\arrowtyped`\arrowtypee;\height>}}
 
\def\Atrianglepairp<#1>[#2;#3`#4`#5`#6`#7]{{
\settripairparms[#1]
\settokens`#2`
\width=\height                                
\diagram
\putAtrianglepairp                            
<\arrowtypea`\arrowtypeb`\arrowtypec`
\arrowtyped`\arrowtypee;\height>
(0,0)[{#2};#3`#4`#5`#6`{#7}]
\enddiagram
}}
 
\def\putVtrianglep<#1>(#2,#3)[#4`#5`#6;#7`#8`#9]{{%
\settriparms[#1]%
\xpos=#2 \ypos=#3
\advance\ypos by\height
{\multiply\height by2
\puthmorphism(\xpos,\ypos)[#4`#5`{#7}]{\height}{\arrowtypea}a}%
\putmorphism(\xpos,\ypos)(1,-1)[`#6`{#8}]{\height}{\arrowtypeb}l%
\advance\xpos by\height
\advance\xpos by\height
\putmorphism(\xpos,\ypos)(-1,-1)[``{#9}]{\height}{\arrowtypec}r%
}}
 
\def\putVtriangle{\@ifnextchar <{\putVtrianglep}{\putVtrianglep
   <\arrowtypea`\arrowtypeb`\arrowtypec;\height>}}
\def\Vtriangle{\@ifnextchar <{\Vtrianglep}{\Vtrianglep
   <\arrowtypea`\arrowtypeb`\arrowtypec;\height>}}
\def\Vtrianglep<#1>[#2`#3`#4;#5`#6`#7]{{
\settriparms[#1]
\width=\height                                 
\diagram
\putVtrianglep<\arrowtypea`\arrowtypeb`
\arrowtypec;\height>
(0,0)[#2`#3`#4;#5`#6`{#7}]
\enddiagram
}}
 
\def\putVtrianglepairp<#1>(#2)[#3;#4`#5`#6`#7`#8]{{
\settripairparms[#1]%
\setpos(#2)%
\settokens`#3`%
\advance\ypos by\height
\putmorphism(\xpos,\ypos)(1,-1)[`\tokend`{#6}]{\height}{\arrowtypec}l%
\puthmorphism(\xpos,\ypos)[\tokena`\tokenb`{#4}]{\height}{\arrowtypea}a%
\advance\xpos by\height
\puthmorphism(\xpos,\ypos)[\phantom{\tokenb}`\tokenc`{#5}]%
{\height}{\arrowtypeb}a%
\putvmorphism(\xpos,\ypos)[``{#7}]{\height}{\arrowtyped}m%
\advance\xpos by\height
\putmorphism(\xpos,\ypos)(-1,-1)[``{#8}]{\height}{\arrowtypee}r%
}}
 
\def\putVtrianglepair{\@ifnextchar <{\putVtrianglepairp}{\putVtrianglepairp%
    <\arrowtypea`\arrowtypeb`\arrowtypec`\arrowtyped`\arrowtypee;\height>}}
\def\Vtrianglepair{\@ifnextchar <{\Vtrianglepairp}{\Vtrianglepairp%
    <\arrowtypea`\arrowtypeb`\arrowtypec`\arrowtyped`\arrowtypee;\height>}}
\def\Vtrianglepairp<#1>[#2;#3`#4`#5`#6`#7]{{
\settripairparms[#1]
\settokens`#2`
\diagram
\putVtrianglepairp                             
<\arrowtypea`\arrowtypeb`\arrowtypec`
\arrowtyped`\arrowtypee;\height>
(0,0)[{#2};#3`#4`#5`#6`{#7}]
\enddiagram
}}

\def\putCtrianglep<#1>(#2,#3)[#4`#5`#6;#7`#8`#9]{{%
\settriparms[#1]%
\xpos=#2 \ypos=#3
\advance\ypos by\height
\putmorphism(\xpos,\ypos)(1,-1)[``{#9}]{\height}{\arrowtypec}l%
\advance\xpos by\height
\advance\ypos by\height
\putmorphism(\xpos,\ypos)(-1,-1)[#4`#5`{#7}]{\height}{\arrowtypea}l%
{\multiply\height by 2
\putvmorphism(\xpos,\ypos)[`#6`{#8}]{\height}{\arrowtypeb}r}%
}}
 
\def\putCtriangle{\@ifnextchar <{\putCtrianglep}{\putCtrianglep
    <\arrowtypea`\arrowtypeb`\arrowtypec;\height>}}
\def\Ctriangle{\@ifnextchar <{\Ctrianglep}{\Ctrianglep
    <\arrowtypea`\arrowtypeb`\arrowtypec;\height>}}
\def\Ctrianglep<#1>[#2`#3`#4;#5`#6`#7]{{
\settriparms[#1]
\width=\height                               
\diagram
\putCtrianglep<\arrowtypea`\arrowtypeb`
\arrowtypec;\height>
(0,0)[#2`#3`#4;#5`#6`{#7}]
\enddiagram
}}                                           
\def\putDtrianglep<#1>(#2,#3)[#4`#5`#6;#7`#8`#9]{{%
\settriparms[#1]%
\xpos=#2 \ypos=#3
\advance\xpos by\height \advance\ypos by\height
\putmorphism(\xpos,\ypos)(-1,-1)[``{#9}]{\height}{\arrowtypec}r%
\advance\xpos by-\height \advance\ypos by\height
\putmorphism(\xpos,\ypos)(1,-1)[`#5`{#8}]{\height}{\arrowtypeb}r%
{\multiply\height by 2
\putvmorphism(\xpos,\ypos)[#4`#6`{#7}]{\height}{\arrowtypea}l}%
}}
 
\def\putDtriangle{\@ifnextchar <{\putDtrianglep}{\putDtrianglep
    <\arrowtypea`\arrowtypeb`\arrowtypec;\height>}}
\def\Dtriangle{\@ifnextchar <{\Dtrianglep}{\Dtrianglep
   <\arrowtypea`\arrowtypeb`\arrowtypec;\height>}}
\def\Dtrianglep<#1>[#2`#3`#4;#5`#6`#7]{{
\settriparms[#1]
\width=\height                              
\diagram
\putDtrianglep<\arrowtypea`\arrowtypeb`
\arrowtypec;\height>
(0,0)[#2`#3`#4;#5`#6`{#7}]
\enddiagram
}}                                          
\def\setrecparms[#1`#2]{\width=#1 \height=#2}%
 
\def\recursep<#1`#2>[#3;#4`#5`#6`#7`#8]{{\m@th
\width=#1 \height=#2
\settokens`#3`
\settowidth{\tempdimen}{$\tokena$}
\ifdim\tempdimen=0pt
  \savebox{\tempboxa}{\hbox{$\tokenb$}}%
  \savebox{\tempboxb}{\hbox{$\tokend$}}%
  \savebox{\tempboxc}{\hbox{$#6$}}%
\else
  \savebox{\tempboxa}{\hbox{$\hbox{$\tokena$}\times\hbox{$\tokenb$}$}}%
  \savebox{\tempboxb}{\hbox{$\hbox{$\tokena$}\times\hbox{$\tokend$}$}}%
  \savebox{\tempboxc}{\hbox{$\hbox{$\tokena$}\times\hbox{$#6$}$}}%
\fi
\ypos=\height
\divide\ypos by 2
\xpos=\ypos
\advance\xpos by \width
\bfig
\putCtrianglep<-1`1`1;\ypos>(0,0)[`\tokenc`;#5`#6`{#7}]%
\puthmorphism(\ypos,0)[\tokend`\usebox{\tempboxb}`{#8}]{\width}{-1}b%
\puthmorphism(\ypos,\height)[\tokenb`\usebox{\tempboxa}`{#4}]{\width}{-1}a%
\advance\ypos by \width
\putvmorphism(\ypos,\height)[``\usebox{\tempboxc}]{\height}1r%
\efig
}}
 
\def\recurse{\@ifnextchar <{\recursep}{\recursep<\width`\height>}}
 
\def\puttwohmorphisms(#1,#2)[#3`#4;#5`#6]#7#8#9{{%
\puthmorphism(#1,#2)[#3`#4`]{#7}0a
\ypos=#2
\advance\ypos by 20
\puthmorphism(#1,\ypos)[\phantom{#3}`\phantom{#4}`#5]{#7}{#8}a
\advance\ypos by -40
\puthmorphism(#1,\ypos)[\phantom{#3}`\phantom{#4}`#6]{#7}{#9}b
}}
 
\def\puttwovmorphisms(#1,#2)[#3`#4;#5`#6]#7#8#9{{%
\putvmorphism(#1,#2)[#3`#4`]{#7}0a
\xpos=#1
\advance\xpos by -20
\putvmorphism(\xpos,#2)[\phantom{#3}`\phantom{#4}`#5]{#7}{#8}l
\advance\xpos by 40
\putvmorphism(\xpos,#2)[\phantom{#3}`\phantom{#4}`#6]{#7}{#9}r
}}
 
\def\puthcoequalizer(#1)[#2`#3`#4;#5`#6`#7]#8#9{{%
\setpos(#1)%
\puttwohmorphisms(\xpos,\ypos)[#2`#3;#5`#6]{#8}11%
\advance\xpos by #8
\puthmorphism(\xpos,\ypos)[\phantom{#3}`#4`#7]{#8}1{#9}
}}
 
\def\putvcoequalizer(#1)[#2`#3`#4;#5`#6`#7]#8#9{{%
\setpos(#1)%
\puttwovmorphisms(\xpos,\ypos)[#2`#3;#5`#6]{#8}11%
\advance\ypos by -#8
\putvmorphism(\xpos,\ypos)[\phantom{#3}`#4`#7]{#8}1{#9}
}}
 
\def\putthreehmorphisms(#1)[#2`#3;#4`#5`#6]#7(#8)#9{{%
\setpos(#1) \settypes(#8)
\if a#9 %
     \vertsize{\tempcounta}{#5}%
     \vertsize{\tempcountb}{#6}%
     \ifnum \tempcounta<\tempcountb \tempcounta=\tempcountb \fi
\else
     \vertsize{\tempcounta}{#4}%
     \vertsize{\tempcountb}{#5}%
     \ifnum \tempcounta<\tempcountb \tempcounta=\tempcountb \fi
\fi
\advance \tempcounta by 60
\puthmorphism(\xpos,\ypos)[#2`#3`#5]{#7}{\arrowtypeb}{#9}
\advance\ypos by \tempcounta
\puthmorphism(\xpos,\ypos)[\phantom{#2}`\phantom{#3}`#4]{#7}{\arrowtypea}{#9}
\advance\ypos by -\tempcounta \advance\ypos by -\tempcounta
\puthmorphism(\xpos,\ypos)[\phantom{#2}`\phantom{#3}`#6]{#7}{\arrowtypec}{#9}
}}
 
\def\setarrowtoks[#1`#2`#3`#4`#5`#6]{%
\def\toka{#1}
\def\tokb{#2}
\def\tokc{#3}
\def\tokd{#4}
\def\toke{#5}
\def\tokf{#6}
}
\def\hex{\@ifnextchar <{\hexp}{\hexp<1000`400>}}
\def\hexp<#1`#2>[#3`#4`#5`#6`#7`#8;#9]{%
\setarrowtoks[#9]
\yext=#2 \advance \yext by #2
\xext=#1 \advance\xext by \yext
\bfig
\putCtriangle<-1`0`1;#2>(0,0)[`#5`;\tokb``\tokd]
\xext=#1 \yext=#2 \advance \yext by #2
\putsquare<1`0`0`1;\xext`\yext>(#2,0)[#3`#4`#7`#8;\toka```\tokf]
\advance \xext by #2
\putDtriangle<0`1`-1;#2>(\xext,0)[`#6`;`\tokc`\toke]
\efig
}
\makeatother

\newtheorem{thm}{Theorem}

\newtheorem{prop}[thm]{Proposition}
\newtheorem{cor}[thm]{Corollary}
\newtheorem{lem}[thm]{Lemma}

\newenvironment{rema}[1]{\noindent {\em Remark.} #1}{}
\newenvironment{exa}[1]{\noindent {\em Example.} #1}{}

\newcommand{\ham}{\textup{Ham}}

\usepackage{ifdraft, blindtext, verbatim}
\ifdraft{%
    \usepackage{todonotes}}{ 
    \usepackage[disable]{todonotes} 
}

\begin{document}
\title{Hamiltonian loops on the symplectic one-point blow up}
\author{Andr\'es Pedroza}
\address{Facultad de Ciencias\\
           Universidad de Colima\\
           Bernal D\'{\i}az del Castillo No. 340\\
           Colima, Col., Mexico 28045}
\email{andres\_pedroza@ucol.mx}

\begin{abstract}
We lift a  loop of Hamiltonian diffeomorphisms on a symplectic manifold to loop of  Hamiltonian diffeomorphisms
on the symplectic one-point blow up of the symplectic manifold. Then we use Weinstein's morphism to 
show that the  lifted  loop of Hamiltonian  diffeomorphisms  has infinite order on the
fundamental group of the group of Hamiltonian diffeomorphisms of the blown
up manifold.

\end{abstract}

\thanks{The author was supported by a CONACYT grant CB 2010/151846.}

\maketitle


\section{Introduction}
\label{s:intro}

The rational homotopy type of the group of Hamiltonian diffeomorphisms of 
 the symplectic one-point blow up $(\widetilde M,\widetilde\omega_\rho)$ of
weight $\rho$ is known for only a special class of symplectic manifolds. 
In \cite{abreu-mcduff-topology},
M. Abreu and D. McDuff computed the rational homotopy type of the group of symplectic diffeomorphisms
of the symplectic one-point blow up of $(\mathbb{C}P^2,\omega_\textup{FS})$
 In \cite{Lalonde-Pin},
F. Lalonde and M. Pinsonnault  computed the rational homotopy type of the above
group for the one-point blow up  of $(S^2\times S^2, \omega\oplus\mu\omega)$ for $1\leq\mu \leq2$;
and in \cite{pinso-symplecto} M. Pinsonnault worked out the case of the one-point blow up of
 rational ruled symplectic 4-manifolds; see also  
\cite{AnjosLalondePin}.  
The case of multiple points blown up simultaneously has also been considered.
J.D. Evans  \cite{Evans-symplectic-mapping} considered this case for
 $(\mathbb{C}P^2,3\omega_\textup{FS})$ blown up at 
$3,4$ and $5$ generic points. 
The reason that all the above examples are in dimension $4$, has to do with the  special  behavior
of holomorphic curves in $4$ dimensional symplectic manifolds.
Apart from these cases, only partial information is known about the homotopy type
of $\textup{Ham}(\widetilde M,\widetilde\omega_\rho)$. 
 For example in \cite{mcduff-blowup}, D. McDuff  showed that if the Hurewicz morphism
 $\pi_2(M)\to H_2(M;\mathbb{Q})$ is non trivial then there exists a non trivial morphism
$\pi_2(M)\to\pi_1(\textup{Ham}(\widetilde M,
\widetilde\omega_\rho))$.

In this paper we will focus on determining that 
$\pi_1(\textup{Ham}(\widetilde M,\widetilde\omega_\rho))$ is non trivial
 for some particular class of symplectic manifolds $(M,\omega)$.
Moreover, the way that we show that $\pi_1(\textup{Ham}(\widetilde M,\widetilde\omega_\rho))$ is non trivial
is by considering a particular class of  loops of Hamiltonian diffeomorphisms in 
$\textup{Ham}( M,\omega)$,  lift it to a loop in $\textup{Ham}(\widetilde M,\widetilde\omega_\rho)$
and then use Weinstein's morphism to show that is not null homotopic.
What is surprising  is that in some cases a non constant  null homotopic loop  in 
$\textup{Ham}( M,\omega)$  lifts   to a loop that is not null homotopic in
$\textup{Ham}
(\widetilde M,\widetilde\omega_\rho)$.

To be more precise about our statements, fix a base point  $x_0\in M$ and a symplectic embedding  
$\iota: (B_\rho,\omega_0)\to (M,\omega)$ of the closed ball of radius $\rho$ 
in $(\mathbb{R}^{2n},\omega_0)$ 
such that $\iota(0)=x_0$.  Relative to the embedding $\iota$ we have 
the symplectic one-point blow up $(\widetilde M,\widetilde\omega_\rho)$
at $x_0$ of weight $\rho$ and the blow up map $\pi:\widetilde M\to M.$
In Section \ref{sec:blow} we review
the construction of the symplectic one-point blow up.
Denote by 
$\mathcal{H}_\rho^U$ the subgroup of Hamiltonian diffeomorphisms $\psi$ of 
$( M,\omega)$ such that
\begin{itemize}
\item[a)] $\psi(x_0)=x_0$, and
\item[b)] $\psi$ acts in a $U(n)$-way in a neighborhood of $\iota B_\rho$.
\end{itemize}
(When we say that $\psi$ behaves in a $U(n)$-way, we mean with respect to
the  coordinates induced by the symplectic embedding.)
Let $\mathcal{H}_{\rho,0}^U$  be the connected component of $\mathcal{H}_\rho^U$ that contains 
the identity map
and 
$\Phi_\rho: \mathcal{H}_{\rho}^U\to \textup{Ham}( M,\omega) $  the inclusion morphism.
It is well known that 
 a diffeomorphism $\psi$ that fixes the base point $x_0$ and  behaves in a $U(n)$-way near $\iota B_\rho$
induces a unique diffeomorphism $\widetilde\psi$ of the one-point blow up $\widetilde M$ such that
$\pi\circ\widetilde \psi=\pi\circ \psi$.  In this case we say that $\widetilde\psi$
{\em lifts} $\psi$.
Now consider the symplectic structure in the process of
of lifting diffeomorphisms;
in Section \ref{sec:symplecdiff} we show that $\widetilde \psi$ is symplectic if
$\psi$ is symplectic; moreover  if $\psi \in\mathcal{H}_{\rho,0}^U$
then $\widetilde \psi$ is a  Hamiltonian diffeomorphism of $(\widetilde M,\widetilde\omega_\rho)$.
 This gives rise to  a group morphism
$\Psi_\rho: \mathcal{H}_{\rho,0}^U\to  \textup{Ham}(\widetilde M,\widetilde\omega_\rho) $
that consist of lifting a Hamiltonian $\psi$ of $(M,\omega)$ to a Hamiltonian $\widetilde \psi$
of $(\widetilde M,\widetilde \omega_\rho)$.
Notice that 
elements in the image of $\Phi_\rho$ are the ones that lift
to Hamiltonian diffeomorphisms of $(\widetilde M,\widetilde\omega_\rho)$  via the morphism $\Psi_\rho$. 
The map $\Psi_\rho$ is known to be a homotopy equivalence  in some cases \cite{Lalonde-Pin}; 
for example in the case 
$(S^2\times S^2,\omega\oplus\mu\omega)$ for $\mu \geq1$ and $0<\rho<1$. 
Indeed, 
this is part of the argument that F. Lalonde and M. Pinsonnault use in the computation of 
the rational homotopy
type of the group of Hamiltonian diffeomorphisms of the one-point blow 
up of $(S^2\times S^2, \omega\oplus\mu\omega)$. 

\begin{center}
\begin{picture}(2300,700)(0,0)
\put(20,500){$\pi_1(\mathcal{H}_{\rho,0}^U) \hskip 3.3cm
\pi_1(\Phi_\rho(\mathcal{H}_{\rho,0}^U ))\subset
  \pi_1(\ham(M,\omega))$}
\put(-130,10){$\pi_1(\textup{Ham}(\widetilde M,\widetilde\omega_\rho))$}
\put(1400,460){\vector(-2,-1){700}}
\put(450,540){\vector(1,0){650}}
\put(130,430){\vector(0,-1){260}}
\put(-60,280){\begin{small}$\Psi_{\rho,*}$ \end{small}}
\put(700,600){\begin{small}$\Phi_{\rho,*}$ \end{small}}
\end{picture}
\end{center}

Now consider the induced maps $\Phi_{\rho,*}$ and $\Psi_{\rho,*}$ on fundamental groups. Thus we are
interested in the
image of $\Phi_{\rho,*}: \pi_1(\mathcal{H}_{\rho,0}^U)\to  \pi_1(\textup{Ham}( M,\omega) )$.
Contrary to the case when the lift of a single Hamiltonian diffeomorphism is unique,
the  lift of $\psi=[\{\psi_t\}]\in \textup{Im} (\Phi_{\rho,*})$  to
$\pi_1(\textup{Ham}(\widetilde M,\widetilde \omega_\rho) )$ 
is not unique. 
The way to single out one element in $\pi_1(\textup{Ham}(\widetilde M,\widetilde \omega_\rho) )$ 
when $\psi$ is in the image of $\Phi_{\rho,*}$
is by fixing a representative $\{\psi_t\}$ of $\psi$
in $\mathcal{H}_{\rho,0}^U$. In other words the map $\Phi_{\rho}$
is injective, but the map $\Phi_{\rho,*}$ is not necessarily injective.
Once we fixed a representative,  by Proposition \ref{p:localliftHam}
we obtain a loop $\{\widetilde\psi_t\}$ in $\textup{Ham}(\widetilde M,\widetilde \omega_\rho)$ and 
define the lifted element as 
$\widetilde\psi:=[ \{   \widetilde\psi_t \} ]$. 

The argument we use to show that a Hamiltonian loop is not null homotopic
is by using  Weinstein's morphism  \cite{weinstein-coho},
$$
\mathcal{A}: \pi_1(\textup{Ham}(M,\omega))\to \mathbb{R}/\mathcal{P}(M,\omega).
$$
Here $\mathcal{P}(M,\omega)$ is the period group of $(M,\omega)$. In Section \ref{s:weinstein}
we review Weinstein's morphism. 

Throughout, we have a fixed symplectic embedding
$\iota :(B_\rho,\omega_0)\to (M,\omega) $	such that $\iota(0)=x_0$.
This embedding is used to define the symplectic one-point blow up of
$(M,\omega)$ at $x_0$
 as a coordinate chart about $x_0$ and to define the lift of Hamiltonian
 diffeomorphisms.

\begin{thm}
\label{t:main}
Let $(M,\omega)$ be a closed symplectic manifold
and $\psi$ an element in $ \pi_1(\textup{Ham}( M,\omega) )$
such that $\Phi_{\rho,*}([\{\psi_t\}])=\psi$,
where the loop $\{\psi_t\}$  
is given by the normalized Hamiltonian $H_t$. 
Then 
for $\widetilde\psi=\Psi_{\rho,*}([\{\psi_t\}])$ in  
$\pi_1(\textup{Ham}(\widetilde M,\widetilde \omega_\rho) )$ we have
\begin{eqnarray}
\label{e:Weinsteinonlift}
\mathcal{A}(\widetilde\psi) = 
\left[\mathcal{A}(\psi) +
\frac{1}{\textup{Vol}(\widetilde M,\widetilde\omega_\rho^n)}
\int_0^1 
\int_{\iota B_{\rho}} H_t\, \omega^n dt
 \right]
\end{eqnarray}
in $\mathbb{R}/\mathcal{P}(\widetilde M,\widetilde \omega_\rho)$.
\end{thm}

There are two things to notice about  expression  (\ref{e:Weinsteinonlift})  of $\mathcal{A}(\widetilde\psi) $.
The second term on the right hand side depends on local information of $\psi$
about $x_0$; and it also reflects the choice of the representative of $\psi$ in order to
lift it to $\widetilde \psi$, namely the Hamiltonian
function $H_t.$

In the special case when the normalized
Hamiltonian function $H_t$ of the loop $\psi$  
takes the form
\begin{eqnarray}
\label{e:hamfun}
H_t(z_1,\ldots,z_n) :=-\pi\sum_{j=1}^n m_j |z_j|^2 + c_t
\end{eqnarray}
on $\iota B_\rho$ where
$m_1,\ldots, m_n\in\mathbb{Z}$ and $c_t\in \mathbb{R}$, Eq. (\ref{e:Weinsteinonlift})  
can be rewritten as
\begin{equation}
\label{e:atilde}
\mathcal{A}(\widetilde\psi)  =\left[\mathcal{A}(\psi) -
\frac{m_1+\cdots +m_n}{(n+1)!} \frac{\pi^{n+1}\rho^{2n+2}}{ \textup{Vol}( M,\omega_\rho^n)- \pi^n\rho^{2n}  } 
 + \frac{ C\, \pi^{n}\rho^{2n}}{\textup{Vol}( M,\omega_\rho^n)- \pi^n\rho^{2n}  }
\right]
\end{equation}
in $\mathbb{R}/\mathcal{P}(\widetilde M,\widetilde \omega_\rho)$, where
$C=\int_0^1 c_t \,dt$. 


A loop   $\psi$ 
in $\textup{Ham}( M,\omega)$ based at the identity map is called  {\em $\iota$-circle
loop} if  on $\iota B_\rho$  
 the corresponding normalized Hamiltonian takes
the form of  Eq. (\ref{e:hamfun}).
Denote by $n(\psi)$ the order of 
$\mathcal{A}([\psi])$   in $\mathbb{R}/\mathcal{P}( M,\omega)$ 

\begin{thm}
\label{t:rank}
Let $(M,\omega)$ be a closed symplectic manifold 
such that $\omega$ is rational.  
If    $\psi_1,\ldots,\psi_{k}$ are $\iota$-circle
loops in $\textup{Ham}( M,\omega)$  and 
$\{n(\psi_1),\ldots ,n(\psi_k)\}$ are pairwise relative prime,
then
$$
\textup{rank } \pi_1(\textup{Ham}(\widetilde M,\widetilde \omega_{\rho_0}) ) \geq k
$$
 for some small $\rho_0<\rho$. Furthermore, the classes of the lifted loops 
$[\widetilde\psi_1],\ldots,[\widetilde\psi_{k}]$ 
generate an abelian subgroup  of  rank $k$ of
$\pi_1(\textup{Ham}(\widetilde M,\widetilde \omega_{\rho_0}) )$.
\end{thm}

The idea behind the proof of Theorem \ref{t:rank}, is that from expression
(\ref{e:atilde})  of $\mathcal{A}(\widetilde \psi)$
we  obtain  a polynomial in $\pi\rho^2$ with rational coefficients. Hence the hypothesis
that the symplectic form  must be rational. 
Then the fact that  $\mathcal{A}(\widetilde \psi)$ has infinite order in $\pi_1(\textup{Ham}
(\widetilde M,\widetilde \omega_{\rho}) )$, 
is equivalent to the fact that $\pi\rho^2$ is not a root of this polynomial. 
Hence  the value $\rho_0$ in Theorem \ref{t:rank} is subject to the condition that 
$\pi\rho^2_0$ must be a transcendental number.

The most common examples of Hamiltonian loops are  Hamiltonian $S^1$-actions. 
Recall that the fixed point set of a Hamiltonian circle action on a closed  symplectic manifold
is non empty.
Hence if $\psi$ is a Hamiltonian circle action on $(M,\omega)$, then by blowing up a fixed point
the above result guarantees that  
$\pi_1(\textup{Ham}(\widetilde M,\widetilde \omega_\rho) )$ has positive rank
for some values of $\rho$ as long as $\mathcal {A}([\psi])$ has finite order. 
This is true  for $(\mathbb{C}P^n,\omega)$, where the symplectic form
is normalized to be rational. Hence by Theorem \ref{t:rank}, the rank of
$\pi_1(\textup{Ham}(\widetilde{ \mathbb{C}P^n},\widetilde \omega_\rho) )$
is greater than or equal to one. 
The results of D. McDuff in \cite{mcduff-blowup} already imply that
the rank of $\pi_1(\textup{Ham}(\widetilde{ \mathbb{C}P^n},\widetilde \omega_\rho) )$
is greater than or equal to one;  we provide an alternative  solution and  show that such    
element of infinite order is induced from an element 
 in $\pi_1(\textup{Ham}(\mathbb{C}P^n,\omega))$ of finite order.

\begin{cor}
\label{c:s1}
Let $(M,\omega)$ be a closed symplectic manifold such that $\omega$
is rational. If $\psi$ is a non trivial  Hamiltonian circle action on $(M,\omega)$
such that $\mathcal{A}([\psi])$ has finite order, then for some small $\rho$
the rank of 
$\pi_1(\textup{Ham}(\widetilde M,\widetilde \omega_\rho) )$ 
is positive.
\end{cor}

\begin{rema}
The conclusions of Theorem \ref{t:rank}  and Corollary \ref{c:s1}
are also valid  for the group      $\pi_1(\mathcal{H}^U_{\rho,0})$; since the elements 
guaranteed by these results are
induced by the map $\Phi_{\rho_*}: \pi_1( \mathcal{H}^U_{\rho,0} )\to
\pi_1(\textup{Ham}(\widetilde M,\widetilde \omega_\rho) )$.
\end{rema}
\medskip

We conjectured that for any closed symplectic manifold $(M,\omega)$,
the rank of $\pi_1(\textup{Ham}(\widetilde M,\widetilde \omega_\rho) )$
must be positive for $\rho$ small.  Note that the restriction on the weight of
the blow up is important. 
For, from the work of J. D. Evans \cite[Theorem 1.3]{Evans-symplectic-mapping}
it follows that the group of Hamiltonian diffeomorphisms  of $(\mathbb{C}P^2,\omega)$
blown up at four generic points is contractible. In this example the weight at each 
point is the same and its value is determined so that the resulting blow up manifold is monotone.

More intriguing  is to know if for every positive integer 
$k$ there exists a closed symplectic 4-manifold such that the rank of 
$\pi_1(\textup{Ham}( M,\omega) )$ is precisely
 $k.$

\medskip
\medskip

The methods mentioned above, lifting Hamiltonian diffeomorphisms and the relation between
$\mathcal{A}(\psi)$ and $\mathcal{A}(\widetilde\psi)$, also work in the case when $k$
points are blown up simultaneously.
As well, in  the non compact case a relation analogous to Eq. (\ref{e:Weinsteinonlift})
can be obtained for Calabi's morphism,
 namely
\begin{eqnarray*}
\label{e:Calabionlift}
\textup{Cal}(\widetilde \psi)=\textup{Cal}(\psi) -
\frac{1}{n!}\int_0^1 \int_{\iota B_\rho}  H_t \,\omega^n dt.
\end{eqnarray*}
Here  $\psi$ is a loop in $\textup{Ham}^c(M,\omega)$, $H_t$ its Hamiltonian function with
 compact support and   $\widetilde\psi$ a lift of the loop.
As in Theorem \ref{t:main}, the lift $\widetilde\psi$ is induced by the representative 
of $\psi$ given by the Hamiltonian $H_t.$

Finally, we make some comments on our notation. In order to simplify notation we use $\psi$
 to denote either a loop 
$\{\psi_t\}_{0\leq t\leq 1}$ of diffeomorphisms based at the identity map, or an
 element $[\{\psi_t\}_{0\leq t\leq 1}]$ in the fundamental group or a single diffeomorphism. 
From context it will be clear which of these three objects $\psi$
stands for.

\section{The symplectic blow up}
\label{sec:blow}

In this section we review the symplectic one-point blow up of a manifold,  with the intention of setting up
notation that will be used throughout the paper.

To that end,  first consider the blow up of $\mathbb{C}^n$ at the origin $\Phi: \widetilde{\mathbb{C}^n}\to
\mathbb{C}^n$,
where $$
\widetilde{\mathbb{C}^n}:=\{ (z,\ell) : z\in \mathbb{C}^n, \ell\in \mathbb{C}P^{n-1}
\textup{ and } z\in \ell\}
$$
and $\Phi(z,\ell)=z.$ 
Recall that $\widetilde{\mathbb{C}^n}$ can also be identified with the tautological 
line bundle $pr:\widetilde{\mathbb{C}^n}\to \mathbb{C}P^{n-1}$, where $pr(z,\ell)=\ell$.
For the closed ball $B_{r}\subset \mathbb{C}^n$ of radius $r$ centred at the origin, set  $L_{r}:=\Phi^{-1}(B_{r})$.
Let $(M,\omega)$ be a symplectic manifold, 
$\omega_0$ is the standard symplectic form on $\mathbb{C}^n$
and $\iota:(B_{r},\omega_0)\to (M,\omega)$  a 
symplectic embedding such that $\iota(0)=x_0$. Then 
as a smooth manifold,   the
blow up of $M$ at $x_0$ is defined as
$$
\widetilde M:= \left( M\setminus \{x_0\} \right)
\cup L_{r}/\sim,
$$
where $x=\iota(z)\in \iota B_{r}\subset M\setminus \{x_0\}$ is identified with $\Phi^{-1}(z)$ for $z\neq 0$.
The projection map $\pi: \widetilde M\to M$, is such that $\pi^{-1}(x_0)=E$
is the exceptional divisor and it induces a diffeomorphism $\widetilde M\setminus E
\to M\setminus\{x_0\}$.

As for the symplectic form on the blow up manifold, first we note that 
$\widetilde{\mathbb{C}^n}$ carries a family  of  K\"ahler forms
\begin{eqnarray}
\label{e:symprho}
\omega(\rho):=\Phi^*(\omega_0)+\rho^2 pr^*(\omega_{\textup{FS}})
\end{eqnarray}
where $\rho>0$ and
the Fubini-Study form on $(\mathbb{C}P^{n-1},\omega_{\textup{FS}})$ is normalized so that
the area of any line is $\pi$.
In order to define a symplectic form
on the blow up manifold $\widetilde M$,   let $\rho<r$;
then the symplectic form $\omega(\rho)$
on $L_{r}$ is perturbed so that near the boundary of $L_{r}$ agrees with the canonical
symplectic form $\omega_0$.  
Let $\beta_\rho:[0,r]\to [\rho, r]$ be defined as
$$
\beta_\rho(s) :=
\left\{
	\begin{array}{ll}
		\sqrt{\rho^2+s^2} & \mbox{for } 0\leq s\leq \delta \\
		s & \mbox{for  } r-\delta \leq s\leq r
	\end{array}
\right.
$$
and on interval $[\delta,r-\delta]$ is defined in any smooth way as long as
 $0<\beta_\rho^\prime(s)\leq 1$
for $0<s\leq r-\delta$.
Define the diffeomorphism
 $F_\rho:L_{r}\setminus E\to  B_{r}\setminus B_{\rho}$    as
$$
F_\rho(z):= \beta_\rho(|z|)\frac{z}{|z|}
$$
and set  $\widetilde\omega(\rho):=F_\rho^*(\omega_0)$. So defined $\widetilde\omega(\rho)$
is a symplectic form such that it
equals $\omega_0$ on $L_{r}\setminus L_{r-\delta}$ and $\omega(\rho)$ on $L_{\delta}$. 
We call $(L_{r}, \widetilde\omega(\rho))$
the local model of the symplectic blow up. 
Now we can define a symplectic form on the blow up manifold. 
The symplectic form 
of weight 
$\rho<r$ on $\widetilde M$ is defined as
$$
\widetilde \omega_\rho :=
\left\{
	\begin{array}{ll}
		\omega  & \mbox{on } \pi^{-1}\left(M \setminus \iota B_{\sqrt{\rho^2+\delta^2}}\right)\\
		\widetilde\omega(\rho) & \mbox{on  } L_{r}.
	\end{array}
\right.
$$
For further details on the symplectic blow up see \cite{ms} and \cite{msjholo}. The above observations are summarized 
in the next proposition.

\begin{prop}
\label{p:propofblow}
Let $\iota:(B_{r},\omega_0)\to (M,\omega)$ be a symplectic embedding such that $\iota(0)=x_0$,
and $(\widetilde M,\widetilde\omega_\rho)$ the symplectic blow up of weight $\rho<r$.
Then 
\begin{enumerate}
 \item $\pi: \widetilde M\setminus E\to M\setminus \{x_0\}$ is a diffeomorphism, 
 \item $\pi^{*}(\omega)=\widetilde \omega_\rho$ on $\pi^{-1}(M\setminus \iota B_{r})$, and
 \item the area of any line in $E$ is $\rho^2\pi$.
\end{enumerate}
\end{prop}


\section{Symplectic and Hamiltonian diffeomorphisms on the blow up}
\label{sec:symplecdiff}

In order to lift a symplectic diffeomorphism $\psi$ on $(M,\omega)$
to a symplectic diffeomorphism $\widetilde\psi$ on $(\widetilde M,
\widetilde \omega_\rho)$, that is in order for the relation $\pi\circ\widetilde\psi=\psi\circ \pi$ to hold,
we must focus on the behavior of $\psi$ on the  embedded ball $\iota B_{r}\subset M$.
A necessary condition to lift $\psi$ is that it must  
map the boundary of $\iota B_{\rho}$ to it self. This is so because the relation
$\pi\circ\widetilde\psi=\psi\circ \pi$ implies that $\widetilde \psi$ maps the 
divisor to itself. Hence, $\psi$ maps $\iota B_{\rho}$ to itself.

As expected the problem of lifting a symplectic diffeomorphism on $M$ to a diffeomorphism 
on the blow up is of  local nature. For that matter
we consider $\psi$ as a symplectic diffeomorphism of $(B_{r},\omega_0)$ 
such that $\psi(0)=0$.   
Further  assume that  
$$
\psi: (B_{r},\omega_0) \to (B_{r},\omega_0)
$$
is given by unitary linear map $\psi=A\in U(n)$.
In this case we
define $\widetilde  \psi:L_{r}\to L_{r}$ by
\begin{eqnarray}
\label{e:deflift}
\widetilde  \psi (z,\ell)=(A(z), A(\ell)).
\end{eqnarray}
Recall  the  classification theorem 
of several complex variables of Cartan \cite{rudin}; a holomorphic map on
$\mathbb{C}^n$ that
maps the ball to it self and fixes the origin must be given by a unitary matrix.  
\begin{lem}
\label{l:symplift}
The map  $\widetilde  \psi$ defined in (\ref{e:deflift})  preserves the symplectic form $\widetilde\omega_\rho$.
\end{lem}
\begin{proof}
From the definitions of $F_\rho$ and $\widetilde\psi$
 we have that $F_\rho \circ   \widetilde\psi= \psi\circ F_\rho$
on $L_{r}\setminus E$. Since  $\widetilde\omega_\rho=F^*_\rho(\omega_0)$, then
$$
(\widetilde  \psi)^*(\widetilde\omega_\rho)=
(\widetilde  \psi)^* \circ F^*_\rho(\omega_0)= F^*_\rho\circ  \psi^*(\omega_0) =
\widetilde\omega_\rho
$$
on $L_{r}\setminus E$.
Finally since $A\in U(n)$, then $\widetilde\psi$  preserves the K\"ahler
form $\omega(\rho)$. In particular the symplectic form $\widetilde\omega_\rho$
on $E$.
\end{proof}

We say that a  symplectic diffeomorphism $\psi$ of $(M,\omega)$  is {\em liftable}
to  $(\widetilde M,\widetilde\omega_\rho)$ if
\begin{itemize}
 \item $\psi(\iota B_{r})=\iota B_{r}$, and
 \item $\iota^{-1}\circ \psi\circ \iota: B_{r}\to B_{r}$ is given by a unitary matrix 
\end{itemize}
where $\rho<r.$ 
This is exactly the description of $\mathcal{H}_{\rho}^U$ given in Section \ref{s:intro}
for the case when $\psi$ is Hamiltonian.
Thus if $\psi$ admits a lift, by Lemma  \ref{l:symplift} we have that $\widetilde\psi$ is a symplectic
diffeomorphism of $(\widetilde M,\widetilde\omega_\rho)$.
It is important to note that the above definition depends on the symplectic 
embedding $\iota: (B_{r},\omega_0)\to (M,\omega)$. Thus from now on we  fix a symplectic 
embedding 
and the lifted diffeomorphisms 
will be  with respect to it.

Now we take into account the problem determining that the lift of a Hamiltonian diffeomorphism
is Hamiltonian. Again, we focus on the local picture.
 Thus let $\psi:B_{r}\to B_{r}$ be liftable and Hamiltonian
and assume  that there is a Hamiltonian path $\{\psi_t\}$,  with $\psi_0=1$, $\psi_1=\psi$
and $\psi_t$ liftable for each $t$. 
Let  $H_t:(B_{r},\omega)\to \mathbb{R}$ and $X_t$ be the Hamiltonian function and time-dependent
vector field 
induced by   the  path  $\{\psi_t\}$. Since the path $\{\psi_t\}$ is liftable, it is actually a path
in $U(n)$; hence $X_t$ is tangent to the sphere centered at the origin. As for the Hamiltonian
function we have the following.

\begin{lem}
\label{l:hamfunc}
Let $\psi_t :(B_{r},\omega_0)\to (B_{r},\omega_0)$ as above; that is, a path of unitary matrices 
starting at the identity matrix. Then 
$H_t(z)=H_t(\lambda z)$ for $z\in B_r$ and $\lambda \in S^1$. 
\end{lem}
\begin{proof}	
Denote by $X_t$ the time-dependent vector field of 
the path $\{\psi_t\}$. For $\lambda\in S^1$, let $\phi_\lambda:B_r\to B_r$ be matrix multiplication
by $\lambda I$. Since $\phi_\lambda$ is in the center of $U(n)$,
\begin{eqnarray*}
X_t\circ \phi_\lambda =\left.\frac{d}{ds}\right|_{s=t} \psi_s\circ \psi_t ^{-1}\circ  \phi_\lambda
=\left.\frac{d}{ds}\right|_{s=t} \phi_\lambda \circ \psi_s\circ \psi_t ^{-1} = (\phi_\lambda)_* X_t.
\end{eqnarray*}
 
Therefore
$$
d (H_t\circ \phi_\lambda)=\omega_0(X_t, (\phi_\lambda)_*( \cdot)) =
(\phi_\lambda) ^*\omega_0 (X_t,\cdot)=d H_t.
$$
Note that  both functions   $H_t$ and $H_t\circ \phi_\lambda$ agree at the origin, thus
 $H_t(z)=H_t(\lambda z)$.
\end{proof}

Since $\psi_t$ is liftable, we have a symplectic path $\{\widetilde\psi_t\}$ on $(L_{r},\widetilde\omega_\rho)$ 
that starts
at the identity and ends at $\widetilde{\psi}_1=\widetilde\psi$. Moreover if $\widetilde X_t$ is the vector field 
induced by $\{\widetilde\psi_t\}$, 
then  we have that $\pi_*(\widetilde X_t)=X_t$
since $\widetilde \psi_t$ is the lift of $\psi_t$.
Now define the function $\widetilde H_t:(L_{r},\widetilde\omega_\rho)\to \mathbb{R}$
as 
\begin{eqnarray}
 \widetilde H_t(z,\ell) :=
\left\{
	\begin{array}{ll}
		H_t\circ F_\rho(z) & \textup{if } (z,\ell)\in L_r\setminus E \\
		H_t\left(  \frac{\rho}{|w|} w\right) & \textup{if $z=0$ and $[w]=\ell$}.
	\end{array}
\right.
\end{eqnarray}
It follows by Lemma $\ref{l:hamfunc}$ that $\widetilde H_t$ is well-defined and smooth.
That is, is independent of the representative of $\ell$ when evaluated at points in the exceptional divisor.
\begin{prop}
\label{p:localliftHam}
The Hamiltonian function $\widetilde H_t:(L_{r},\widetilde\omega_\rho)\to \mathbb{R}$ defined
above induces the path of   Hamiltonian diffeomorphisms $\{\widetilde\psi_t\}$ that is the lift of
the path $\{\psi_t\}$. 
Moreover $\widetilde X_t$
is such that $\pi_*(\widetilde X_t)=X_t.$
\end{prop}
\begin{proof}
We already showed that  the vector fields $\widetilde X_t$ and $X_t$ are related
by the blow up map. It only remains to show that $
\iota(\widetilde X_t)\widetilde \omega_\rho= d\widetilde H_t$. 
First note that
$$
(F_{\rho})_{*,x}(X)= \beta_\rho(x) X +d\beta_\rho(X)x.
$$
Since $\beta_\rho$ is radial, the kernel of $d\beta_\rho$ agrees with the tangent space
to the sphere centred at the origin. Now $\psi_t$ is defined by a unitary matrix, thus
outside the origin and $E$, the vector fields $X_t$ and $\widetilde X_t$ lie in the tangent space
of the sphere. Thus 
$F_{\rho,*}(\widetilde X_t)= \beta_\rho  X_t$ and 
\begin{eqnarray*}
 \widetilde \omega_\rho (\widetilde X_t, \cdot) 
 &=&  
 F_\rho^*( \omega_0) (\widetilde X_t, \cdot  ) \\
&=&  
 \omega_0 (F_{\rho,*}\widetilde X_t, F_{\rho,*} (\cdot) ) \\
&=&  
 \omega_0 (\beta_\rho\cdot X_t, F_{\rho,*} (\cdot)) \\
&=&  
 \beta_\rho \, \omega_0 ( X_t, \beta_\rho^{-1}\cdot F_{\rho,*} (\cdot)) \\
&=&  
 \beta_\rho (dH_t)\circ \beta_\rho^{-1}\cdot F_{\rho,*} \\
 &=&  
d(H_t\circ F_\rho).
 \end{eqnarray*}
on  $L_{r}\setminus E$.
\end{proof}

Hence if $\{\psi_t\}$ is a Hamiltonian path on $(M,\omega)$ with Hamiltonian
function $H_t$ and each $\psi_t$ is liftable, that is a path in
$\mathcal{H}^U_{\rho,0}$,
then the lift $\{\widetilde\psi_t\}$
is a Hamiltonian path with Hamiltonian function $\widetilde H_t: 
(\widetilde M, \widetilde\omega_\rho)\to \mathbb{R}$ given by
\begin{eqnarray}
\label{e:hamonblow}
 \widetilde H_t(x) :=
\left\{
	\begin{array}{ll}
		H_t\circ \pi(x)  & \textup{if } \pi(x)\notin \iota B_{r} \\
		H_t  \circ \iota  \circ F_\rho \circ  \iota^{-1}\circ \pi(x)  &
 \textup{if } \pi(x)\in \iota B_{r}\setminus\{x_0\} \\
		H_t\left(  \frac{\rho}{|w|} w\right) &  \textup{if  $x=[x_0,\ell]\in E$ and $[w]=\ell$}.
	\end{array}
\right.
\end{eqnarray}
Thus Proposition \ref{p:localliftHam} can be stated in global terms.

\begin{prop}
\label{p:globalliftHam}
Let $\{\psi_t\}$ be a path of Hamiltonian diffeomorphisms in $\mathcal{H}^U_{\rho,0}$ with Hamiltonian
function $H_t$. Then the lifted
path $\{\widetilde\psi_t\}$ is a Hamiltonian path on $(\widetilde M,\widetilde{\omega}_\rho)$
with Hamiltonian function $\widetilde H_t$ given by (\ref{e:hamonblow}).
\end{prop}
\begin{rema}
The Hamiltonian diffeomorphism on $(\widetilde M,\widetilde{\omega}_\rho)$
induced by the map $H_t\circ \pi$, is {\em not} the one that lifts the Hamiltonian diffeomorphism of
the base manifold. 
Most importantly to our interest,  if $H_t$ generates a loop  of 
Hamiltonian diffeomorphisms in $(M,\omega)$, then $H_t\circ \pi$ induces a path and 
{\em not}  a loop of Hamiltonian diffeomorphisms on $(\widetilde M,\widetilde{\omega}_\rho)$; the time-one
Hamiltonian diffeomorphism of $H_t\circ \pi$ is not the  identity map. Notice also that $H_t\circ \pi$
is independent of $\rho$, whereas $\widetilde H_t$ depends on $F_\rho$.
\end{rema}

\smallskip
There are two typical examples of symplectic diffeomorphisms that are liftable. The first
class of examples is when the support of $\psi$ is disjoint from $\iota B_{r}$. 
In this case the matrix representation of $\psi$ on $\iota B _{r}$ is the
identity matrix.
Another example is a circle action with $x_0$
a fixed point of the action. 
In this case there is a Darboux chart
about $x_0$ so that the action can be described by a loop of unitary matrices.

\medskip
\begin{exa} 
In this example we see how the definition of the Hamiltonian function
$\widetilde H_t$ given in  (\ref{e:hamonblow}) coincides with the natural Hamiltonian function
on $\widetilde {\mathbb{C}^{n}}$, in the case of a linear circle action on $\mathbb{C}^n$.
To that end,  consider a linear circle action  on $(\mathbb{C}^n,\omega_0)$ with Hamiltonian function
 $H:\mathbb{C}^n\to \mathbb{R}$ given by
$$
H(z_1,\ldots,z_n):=-\pi\sum_{j=1}^n m_j |z_j|^2,
$$
where $m_1,\ldots, m_n\in \mathbb{Z}$ and  $\omega_0(X,\cdot)=dH$.
Since the action is linear, it  induces
a Hamiltonian circle action on $(\mathbb{C}P^{n-1},\omega_{\textup{FS}})$ with Hamiltonian function
$$
H^\prime([z_1:\cdots:z_n]):=-\pi\sum_{j=1}^n m_j  \frac{|z_j|^2}{|z|^2}
$$
and 
$\omega_{FS}(X^\prime,\cdot)=dH^\prime$.

Thus we have a circle action on $\mathbb{C}^n \times \mathbb{C}P^{n-1}$
given by the diagonal action. Furthermore $\widetilde {\mathbb{C}^{n}}$ is invariant under
the  action. Recall the symplectic form 
$\omega(\rho)=
\Phi^*(\omega_0)+\rho^2 pr^*(\omega_{\textup{FS}})$ 
on $\widetilde {\mathbb{C}^{n}}$.   Then the circle
action on $(\widetilde {\mathbb{C}^{n}}, \omega(\rho))$ is Hamiltonian, with Hamiltonian function
$H+\rho^2 H^\prime$ restricted to $\widetilde {\mathbb{C}^{n}}$.

Now we compute the Hamiltonian function $\widetilde H$
on a small neighborhood $U$ of the exceptional divisor, following the definition given by Eq.
 (\ref{e:hamonblow}), 
Recall that the symplectic form $\widetilde \omega_\rho$ on $\widetilde {\mathbb{C}^{n}}$
equals $\omega(\rho)$ on $U$. Then for 
$(z,[z_1:\cdots :z_n])\in U\setminus E$ we have
\begin{eqnarray*}
\widetilde H(z,[z_1:\cdots :z_n])&=& H\circ F_\rho(z) \\ 
&=& H\left( \sqrt{\rho^2+{|z|}^2}  \frac{z}{|z|}  \right) =-\pi\frac{\rho^2+{|z|}^2}{{|z|}^2}\sum_{j=1}^n m_j |z_j|^2\\
&=&- \pi\sum_{j=1}^n m_j {|z_j|^2} -\rho^2 \pi \sum_{j=1}^n m_j \frac{|z_j|^2}{|z|^2}. 
\end{eqnarray*}
Now for 
$(0,[w_1:\cdots w_n])$ in the exceptional divisor
\begin{eqnarray*}
\widetilde H(0,[w_1:\cdots :w_n])&=& H\left(  \frac{\rho}{|w|}w  \right)\\
&=& -\pi\sum_{j=1}^n m_j  \rho^2\frac{|w_j|^2}{|w|^2}.
\end{eqnarray*}
That is $\widetilde H=H+\rho^2 H^\prime$ in a small neighborhood  of the exceptional divisor.
\end{exa}
\medskip

The process of blowing up a point has an alternative description than the one presented
in Section \ref{sec:blow}.
Heuristically, the blow up of a point of weight $\rho$  can be described as removing the 
interior of the embedded ball $ B_{\rho}$  and collapsing its boundary
to $\mathbb{C}P^{n-1}$ via the Hopf fibration. 
The next result is a consequence of this fact.

\begin{lem}
\label{l:normalizedHam}
Let $H:(M,\omega)\to \mathbb{R}$ be a smooth function with compact support and 
$\widetilde H: (\widetilde M,\widetilde \omega_\rho)\to \mathbb{R}$
defined  as in  (\ref{e:hamonblow}). Then,
$$
 \int_{\widetilde M} \widetilde H\, \widetilde\omega^n_\rho=
\int_{M} H\,\omega^n -  \int_{\iota B_\rho} H\,\omega^n.  
$$
\end{lem}
\begin{proof}
By Proposition \ref{p:propofblow} the blow up map induces a symplectic 
diffeomorphism between $(\widetilde M\setminus \pi^{-1}(\iota B_{r}),\widetilde\omega_\rho)$
and $(M\setminus \iota B_{r}, \omega)$. Since $\widetilde H=H\circ \pi$ on $\widetilde M\setminus \pi^{-1}(\iota B_{r})$
we get
$$
\int_{\widetilde M} \widetilde H\, \widetilde\omega^n_\rho =
\int_{ M\setminus \iota B_{r}  } H\,\omega^n+
\int_{\pi^{-1}(\iota B_{r})} \widetilde H\, \widetilde\omega^n_\rho.
$$
By the definition of $\widetilde H$  on $\pi^{-1}(\iota B_{r})$, the fact  that 
$F_\rho^*(\omega_0)=\widetilde\omega_\rho$  on 
$\iota B_{r}\setminus \iota B_\rho$
and removing the exceptional divisor from the domain
of the second integral, the claim follows;
\begin{eqnarray*}
\int_{\widetilde M} \widetilde H\, \widetilde\omega^n_\rho&=&
\int_{ M\setminus \iota B_{r}  } H\,\omega^n+
\int_{\pi^{-1}(\iota B_{r})\setminus  E} H\circ F_\rho \,\widetilde\omega^n_\rho\\
&=&
\int_{ M\setminus \iota B_{r}  } H\,\omega^n+
\int_{\iota B_{r}\setminus \iota B_\rho} H\,\omega^n\\
&=&
\int_{M} H\,\omega^n -\int_{\iota B_\rho} H\,\omega^n.
\end{eqnarray*}
\end{proof}

\begin{rema}
Remember that  we fix a symplectic embedding $\iota: (B_{r},\omega_0)\to (M,\omega)$ 
and respect to this embedding we have lifted symplectic and Hamiltonian diffeomorphisms. 
Clearly a different embedding 
might yield a different set of diffeomorphisms that are liftable. Recall from \cite{mcduffpol-packing}, that
if the embeddings are isotopic via symplectic embeddings  then 
the symplectic blow ups are symplectomorphic. 
Since we are interested in the topology of the
group $\textup{Ham}(\widetilde M,\widetilde \omega_\rho)$,
for our purpose it suffices to fix a symplectic embedding and to require 
the unitary condition on diffeomorphisms in a neighborhood of $\iota B_\rho$
and not on all $\iota B_r.$
\end{rema}

\section{Weinstein's morphism on $(\widetilde{M},\widetilde \omega_\rho)$}
\label{s:weinstein}

Now we consider the case of loops in $ \textup{Ham}
(\widetilde{M},\widetilde \omega_\rho)$ when $(M,\omega)$ is a closed
manifold and $\pi \rho^2 < c_G(M,\omega)$, where $c_G(M,\omega)$
stands for the Gromov's width of $(M,\omega)$. As in Section \ref{sec:symplecdiff},
let $x_0\in M$ be a based point 
and $\iota: (B_{r},\omega_0)\to (M,\omega)$  be a fixed symplectic embedding such
that $\iota(0)=x_0$  and $\pi \rho^2<\pi r^2 < c_G(M,\omega)$.

Recall that the period group $\mathcal{P}(M,\omega)$ of $(M,\omega)$
is defined as the image of the pairing $[\omega]\cdot H_2(M;\mathbb{Z})
\to \mathbb{R}$. Weinstein's morphism \cite{weinstein-coho},
$$
\mathcal{A}:\pi_1 (\textup{Ham}(M)) \to \mathbb{R}/ \mathcal{P}( M,\omega)
$$
is defined via the action functional as
$$
\mathcal{A}(\psi)=-\int_{D} u^*( \omega) +\int_{0}^1 H_t(\psi_t(x_0)) dt.
$$
Here $D$ is the unit closed disk and  $u:D\to M$ is a smooth function such that
$u(\partial D)$ is the loop $\{\psi_t(x_0)\}$ and $H_t$ is the 1-periodic Hamiltonian function
induced by  the Hamiltonian  loop $\{\psi_t\}$ subject to the normalized condition
$$
\int_M H_t \, \omega^n=0
$$
for every $t\in [0,1]$.

Remember that  the  dimension of $(M,\omega)$ is greater than two.
Then for the one-point blow up $(\widetilde M,\widetilde\omega_\rho)$,
we have that
$H_2(\widetilde M;\mathbb{Z})\simeq  H_2(M;\mathbb{Z})+\mathbb{Z}\langle L\rangle$
where $L\subset E$ if the class of a line in the exceptional divisor of
$(\widetilde M,\widetilde\omega_\rho)$.
Note also that any class in 
$ H_2(M;\mathbb{Z})$ can be represented by a cycle  away from the embedded ball $\iota B_{r}$.
Hence  $\langle[\omega], c\rangle=
\langle[\widetilde\omega_\rho], \pi^{-1}(c)\rangle$
for any  $c\in H_2(M;\mathbb{Z})$.
By definition of the symplectic form $\widetilde\omega_\rho$ on the blow up, 
$L\subset(\widetilde M,\widetilde\omega_\rho)$
has symplectic area $\rho^2\pi$ and
$$
\mathcal{P}(\widetilde M,\widetilde \omega_\rho)= 
\mathcal{P}( M,\omega) + \mathbb{Z}\langle \pi\rho^2\rangle\subset \mathbb{R}.
$$


Now we give the proofs of the results mentioned at the Introduction.

\begin{proof}[Proof of Theorem. \ref{t:main}]
Let $\{\psi_t\}$ be a loop in $\textup{Ham}(M,\omega)$ 
that is liftable with respect
to the symplectic embedding $\iota: (B_{r},\omega_0)\to (M,\omega)$. That is $\psi_t
\in \mathcal{H}^U_{\rho,0}$ for every $t$.
Fix $p_0\in M$ outside the embedded ball $\iota B_{r}$, since $\psi$
is liftable the loop  $\gamma:=\{\psi_t(p_0)\}$ in $M$ lies  outside the
embedded ball. Hence 
$\widetilde\gamma:=\{\widetilde\psi_t(\pi^{-1}(p_0))\}$ is a loop
in $\widetilde M$ that covers $\gamma$.

Now let $u:D\to M$ 
be a smooth map such that $u(\partial D)=\gamma$. Since $M$ has dimension
greater than two, by the excision theorem for homotopy groups we can
assume that $u(D)$ is disjoint from $\iota B_{r}$. Hence there is a smooth
map $\widetilde u: D\to \widetilde M$, such that $\widetilde u(\partial D)=\widetilde \gamma$
and $\pi\circ\widetilde u=u$.
Since $\pi^*\omega=\widetilde\omega_\rho$ on $\widetilde M\setminus \pi^{-1}(\iota B_{r})$,
we get
$$
\int_{D} u^*(\omega)= \int_{D} \widetilde u^*(\widetilde\omega_\rho ).
$$

Let $H_t:(M,\omega)\to \mathbb{R}$ be the normalized Hamiltonian function induced
by the loop $\psi$. Then by Lemma \ref{l:normalizedHam} the normalized Hamiltonian
of the lifted loop $\widetilde\psi$ is $\widetilde H_t+c_\rho(M,\omega,H_t)$ where 
$\widetilde H_t$ is given by Eq. (\ref{e:hamonblow}) and
$$
c_\rho(M,\omega,H_t):=- \frac{1}{\textup{Vol}(\widetilde M,\widetilde\omega_\rho^n)} 
  \int_{\widetilde M} \tilde H_t \, \tilde{\omega}_\rho^n
= \frac{1}{\textup{Vol}(\widetilde M,\widetilde\omega_\rho^n)}   
  \int_{\iota B_\rho} H_t \,\omega^n.
$$
Hence,
\begin{eqnarray*}
\mathcal{A}(\widetilde\psi) &=&-\int_D \widetilde u^*(\widetilde\omega_\rho)+ 
\int_0^1 (\widetilde H_t+ c_\rho(M,\omega,H_t))(\widetilde\psi_t(\pi^{-1}(p_0)))dt \\
&=& -\int_D u^*( \omega)+ \int_0^1 H_t(\psi_t(p_0))dt + \int_0^1 c_\rho(M,\omega,H_t) dt\\
&=& \left[\mathcal{A}(\psi) + \frac{1}{\textup{Vol}(\widetilde M,\widetilde\omega_\rho^n)}
\int_0^1    \int_{\iota B_\rho} H_t \, \omega^n \,  dt \right].
\end{eqnarray*}

\end{proof}
In the case when the  normalized Hamiltonian function takes the form
\begin{equation}
\label{e:hamfun2}
H_t(z_1,\ldots,z_n) :=-\pi\sum_{j=1}^n m_j |z_j|^2+c_t
\end{equation}
on $\iota B_r$, we have 
\begin{eqnarray*}
\int_0^1    \int_{\iota B_\rho} H_t \, \omega^n \,  dt   
= -(m_1+\cdots +m_n) \frac{\pi^{n+1} \rho^{2n+2}}{ (n+1)!}  + \textup{Vol}( B_\rho,\omega_0^n)\int_0^1 c_t.
\end{eqnarray*}
Since the volume of $(B_\rho, \omega_0^n)$ is $\pi^n\rho^{2n}$,  then in this case 
$\mathcal{A}(\widetilde\psi)$ takes the form
\begin{equation}
\label{e:atilde2}
\mathcal{A}(\widetilde\psi)  =\left[\mathcal{A}(\psi) -
\frac{m_1+\cdots +m_n}{(n+1)!} \frac{\pi^{n+1}\rho^{2n+2}}{ \textup{Vol}( M,\omega_\rho^n)- \pi^n\rho^{2n}  } 
 +    \frac{C \, \pi^{n}\rho^{2n}}{\textup{Vol}( M,\omega_\rho^n)- \pi^n\rho^{2n}  }
\right]
\end{equation}
where $C=\int_0^1 c_t.$

As mentioned at the Introduction, the proof of Theorem \ref{t:rank} relies on some polynomials 
with rational coefficients. In part, we take care of this by assuming that the symplectic form
 must be rational. However some work needs to be done in order 
to guarantee that the  constant $C$ that
appears in Eq. (\ref{e:atilde2}) is in fact a rational number.

Recall  that an $\iota$-circle loop 
$\psi$ in $\ham(M,\omega)$, is a loop of Hamiltonian diffeomorphisms
based at the identity  map such that its normalized Hamiltonian
function on $\iota B_\rho$ takes the form as in  Eq. (\ref{e:hamfun2}).

\begin{lem}
\label{l:crat}
Let $\psi$ be an $\iota$-circle loop in $\ham(M,\omega)$. Then
$$
\mathcal{A}(\psi)=\left[  C\right].
$$
\end{lem}

\begin{proof}
Let $H_t$ be the normalized Hamiltonian function of the loop.
Since the loop $\psi$ is  $\iota$-circle, then $H_t$ is given by Eq.
(\ref{e:hamfun2})
on $\iota B_\rho$. In particular $x_0$ is fixed by each Hamiltonian
in the loop.
Therefore,
$$
\mathcal{A}(\psi) = \left[\int_0^1 H_t(x_0) dt\right] = \left[\int_0^1 c_t dt \right] =\left[  C\right].
$$
\end{proof}

If $ \psi $ is an $\iota$-circle loop in $\textup{Ham}(M,\omega)$,
denote by $K(\psi,x_0)$ the sum of its weights 
$m_1+\cdots +m_n$  at $x_0$.

\begin{proof}[Proof of Theorem \ref{t:rank}]
Since the symplectic form $\omega$ is rational and $(M,\omega)$ is closed, the period group 
$\mathcal{P}(M,\omega)$ is discrete and $V:=\textup{Vol}(M,\omega^n)$ is a rational number.
Moreover 
$\mathcal{P}(M,\omega)=\mathbb{Z}\langle a \rangle$ for some $a\in \mathbb{Q}\setminus\{0\}$.

Let $\rho_0>0$ be such that  $\pi\rho_0^2$ is transcendental and less than the Gromov
width of $(M,\omega)$. Then consider the 
blow up of $(M,\omega)$ at $x_0$ of weight $\rho_0.$
Since $\psi_j$ is a loop in $\mathcal{H}^U_0 $, it follows by Proposition
3.8 that  it can be lifted to a Hamiltonian 
loop $\widetilde \psi_j$ on $(\widetilde M,\widetilde\omega_{\rho_0})$.
By hypothesis, there are  positive integers  $n_j:=n(\psi_j)$
such that $\mathcal{A}([\psi_j])=[a/n_j]$ in 
$\mathbb{R}/\mathbb{Z}\langle a\rangle$ for $1\leq j\leq k$.
Then by Lemma \ref{l:crat}, we have that $\mathcal{A}([\psi_j])=[C_j]=[a/n_j]$.
Hence  by  Eq. (\ref{e:atilde2}) we get
\begin{equation}
\label{e:newexp}
\mathcal{A}([\widetilde{\psi}_j])  =\left[	\frac{a}{n_j}
\left(1 + \frac{ (\pi^{}\rho_0^{2})^{n}}{V- (\pi^{}\rho_0^{2})^{n}  }
 \right) -
\frac{K(\gamma_j,x_0)}{(n+1)!} \frac{(\pi^{}\rho_0^{2})^{n+1}}{ V-
 (\pi^{}\rho_0^{2})^{n}  }  
\right]
\end{equation} 
in $\mathbb{R}/\mathbb{Z}\langle a, \pi \rho_0^2\rangle$.

For $k\in \mathbb{Z}$ non zero
 the expression
$k \mathcal{A}([\widetilde{\psi}_j]) =0$ is equivalent to 
$k \mathcal{A}([\widetilde{\psi}_j]) \in \mathbb{Z}\langle a, \pi \rho_0^2\rangle$;
that after multiplying it  by ${V- (\pi^{}\rho_0^{2})^{n}  }$ 
gives a polynomial  of degree $n+1$ in $\pi \rho_0^2$ with rational coefficients 
equal to zero. Since $\pi \rho_0^2$ is assumed to be a transcendental number, 
$k \mathcal{A}([\widetilde{\psi}_j]) \neq 0$ and  
each
$[\widetilde{\psi}_j]$ has infinite order in 
$\pi_1(\textup{Ham}(\widetilde M,\widetilde \omega_{\rho_0}) )$.
Note that  $\rho_0$ is independent of $j$.

\medskip
The rest of the proof is devoted to show
that 
$[\widetilde{\psi}_1], \ldots, [\widetilde{\psi}_k]$
generate a subgroup of 
$\pi_1(\textup{Ham}(\widetilde M,\widetilde \omega_{\rho_0}) )$ isomorphic to $\mathbb{Z}^k.$
Assume that   $\mathcal{A}([\widetilde\psi_r])$ is in the group
generated by $\{\mathcal{A}([\widetilde\psi_j])| j\neq r\}$.
Then there exist $\alpha_j\in\mathbb{Z}$ such that
\begin{eqnarray}
\label{e:relation}
\mathcal{A}([\widetilde\psi_r]) =  \alpha_1 \mathcal{A}([\widetilde\psi_1])
+\cdots+
\alpha_k\mathcal{A}([\widetilde\psi_k])   \
\end{eqnarray}
in $\mathbb{R}/\mathbb{Z}\langle a, \pi \rho_0^2\rangle$,
 where there is no $r$-term on the right hand side.
Substituting Eq. (\ref{e:newexp}) in each term of Eq. (\ref{e:relation}), we get that  
\begin{multline}
\label{e:long}
a\left( \frac{1}{n_r}-\frac{\alpha_1}{n_1}-\cdots -\frac{\alpha_k}{n_k} \right)
\left(1 +   
{\frac{ (\pi^{}\rho_0^{2})^{n}}{V- (\pi^{}\rho_0^{2})^{n}  } }
 \right) + \\
 \frac{ -K(\gamma_j,x_0)+\alpha_1 K(\gamma_1,x_0) +\cdots  +\alpha_k K(\gamma_k,x_0)  }
 {(n+1)!} \frac{(\pi^{}\rho_0^{2})^{n+1}}{ V-
 (\pi^{}\rho_0^{2})^{n}  }  
\end{multline}
belongs to  $\mathbb{Z}\langle a, \pi \rho_0^2\rangle$. 
In order to handle the above expression, set
\begin{eqnarray*}
\alpha_0:=\left( \frac{1}{n_r}-\frac{\alpha_1}{n_1}-\cdots -\frac{\alpha_k}{n_k} \right) \in\mathbb{Q}
\end{eqnarray*}
and
\begin{eqnarray*}
K_0:= \frac{ -K(\gamma_j,x_0)+\alpha_1 K(\gamma_1,x_0) +\cdots  +\alpha_k K(\gamma_k,x_0)  }
 {(n+1)!}\in \mathbb{Q}.
\end{eqnarray*}
Then expression (\ref{e:long}) takes the form
\begin{eqnarray}
\label{e:long3}
a \alpha_0  \left(1 +   
{\frac{ (\pi^{}\rho_0^{2})^{n}}{V- (\pi^{}\rho_0^{2})^{n}  } }
 \right) +   K_0  \frac{(\pi^{}\rho_0^{2})^{n+1}}{ V-
 (\pi^{}\rho_0^{2})^{n}  } =     a A +  \pi^{}\rho_0^{2} B.
\end{eqnarray}
for some $A,B\in \mathbb{Z}$. After multiplying Eq. (\ref{e:long3})
by ${V- (\pi^{}\rho_0^{2})^{n}  }$, we obtain a polynomial of degree $n+1$ in $\pi\rho_0^2$ with rational 
coefficients that is equal to zero;
$$
(B+K_0)(\pi^{}\rho_0^{2})^{n+1}+ aA (\pi^{}\rho_0^{2})^{n}-B V(\pi^{}\rho_0^{2})+
(a\alpha_0 V+K_0-aAV)=0.
$$

Finally, since $\pi\rho_0^2$ is a transcendental number  and $a$ is the generator
of the period group of $(M,\omega)$, 
$A$ and $B$ must be zero.  Thus $K_0=0$ and $\alpha_0=0$.
Since $\{n_1,\ldots , n_k\}$ 
are pairwise relatively prime,
from   Lemma
\ref{l:num} we get that  $\alpha_0\neq0$. This means that Eq.
(\ref{e:relation})  holds only when the $\alpha_j$'s are zero, that is
$[\widetilde{\psi}_1], \ldots, [\widetilde{\psi}_k]$
generate a subgroup of rank $k$.
\end{proof}

In the last part of the proof of Theorem \ref{t:rank} we used the following fact about  integers.
Only when quoting this lemma, we used the hypothesis  that the  $n_j$'s are relative prime by pairs.

\begin{lem}
\label{l:num}
Let $n_1,\ldots ,n_k$ be integers, such that
\begin{itemize}
\item $n_j\geq 2$
\item $(n_1,n_j)=1$ for all $j>1.$
\end{itemize}
Then for any $\alpha_2,\ldots,\alpha_k \in \mathbb{Z}$, 
$$
\frac{1}{n_1}-\frac{\alpha_2}{n_2}-\dots -
\frac{\alpha_{k-1}}{n_{k-1}}
-\frac{\alpha_k}{n_k} 
$$
is not equal to zero.
\end{lem}
\begin{proof}
Assume that 
$$
\frac{1}{n_1}-\frac{\alpha_2}{n_2}-\dots -
\frac{\alpha_{k-1}}{n_{k-1}}
-\frac{\alpha_k}{n_k} =0,
$$
thus
$$
n_2\cdots n_k = \sum_{j=2}^{k} n_1	\cdots \alpha_j\cdots n_k.
$$
That is $n_2\cdots n_k$ is a multiple of $n_1$, which is not possible.
\end{proof}

\bibliographystyle{acm}
\bibliography{/Users/andres/Dropbox/Documentostex/Ref.bib} 

\begin{thebibliography}{10}

\bibitem{abreu-mcduff-topology}
{\sc Abreu, M., and McDuff, D.}
\newblock Topology of symplectomorphism groups of rational ruled surfaces.
\newblock {\em J. Amer. Math. Soc. 13}, 4 (2000), 971--1009 (electronic).

\bibitem{AnjosLalondePin}
{\sc Anjos, S., Lalonde, F., and Pinsonnault, M.}
\newblock The homotopy type of the space of symplectic balls in rational ruled
  4-manifolds.
\newblock {\em Geom. Topol. 13}, 2 (2009), 1177--1227.

\bibitem{Evans-symplectic-mapping}
{\sc Evans, J.~D.}
\newblock Symplectic mapping class groups of some {S}tein and rational
  surfaces.
\newblock {\em J. Symplectic Geom. 9}, 1 (2011), 45--82.

\bibitem{gromov-psudo}
{\sc Gromov, M.}
\newblock Pseudoholomorphic curves in symplectic manifolds.
\newblock {\em Invent. Math. 82}, 2 (1985), 307--347.

\bibitem{Lalonde-Pin}
{\sc Lalonde, F., and Pinsonnault, M.}
\newblock The topology of the space of symplectic balls in rational
  4-manifolds.
\newblock {\em Duke Math. J. 122}, 2 (2004), 347--397.

\bibitem{mcduff-examplesof}
{\sc McDuff, D.}
\newblock Examples of symplectic structures.
\newblock {\em Invent. Math. 89}, 1 (1987), 13--36.

\bibitem{mcduff-blowup}
{\sc McDuff, D.}
\newblock The symplectomorphism group of a blow up.
\newblock {\em Geom. Dedicata 132\/} (2008), 1--29.

\bibitem{mcduffpol-packing}
{\sc McDuff, D., and Polterovich, L.}
\newblock Symplectic packings and algebraic geometry.
\newblock {\em Invent. Math. 115}, 3 (1994), 405--434.
\newblock With an appendix by Yael Karshon.

\bibitem{ms}
{\sc McDuff, D., and Salamon, D.}
\newblock {\em Introduction to symplectic topology}, second~ed.
\newblock Oxford Mathematical Monographs. The Clarendon Press, Oxford
  University Press, New York, 1998.

\bibitem{msjholo}
{\sc McDuff, D., and Salamon, D.}
\newblock {\em {$J$}-holomorphic curves and symplectic topology}, second~ed.,
  vol.~52 of {\em A. M. S. Colloquium Publications}.
\newblock American Mathematical Society, Providence, RI, 2012.

\bibitem{pinso-symplecto}
{\sc Pinsonnault, M.}
\newblock Symplectomorphism groups and embeddings of balls into rational ruled
  4-manifolds.
\newblock {\em Compos. Math. 144}, 3 (2008), 787--810.

\bibitem{rudin}
{\sc Rudin, W.}
\newblock {\em Function theory in the unit ball of {${\bf C}^{n}$}}, vol.~241
  of {\em Fundamental Principles of Mathematical Science}.
\newblock Springer-Verlag, New York-Berlin, 1980.

\bibitem{weinstein-coho}
{\sc Weinstein, A.}
\newblock Cohomology of symplectomorphism groups and critical values of
  {H}amiltonians.
\newblock {\em Math. Z. 201}, 1 (1989), 75--82.

\end{thebibliography}
\end{document}